\documentclass[12pt,reqno]{amsart}
\advance\hoffset by -0.5 truecm
\usepackage{amsmath,amssymb,amsfonts}
\usepackage{graphics}
\usepackage{latexsym}

\newtheorem{Theorem}{Theorem}[section]
\newtheorem{Lemma}[Theorem]{Lemma}

\newtheorem{Corollary}[Theorem]{Corollary}

\newtheorem{Example}[Theorem]{Example}
\newtheorem{Remark}[Theorem]{Remark}

\begin{document}

\title[Weyl functional on 4-manifolds]{The Weyl functional on 4-manifolds of positive Yamabe invariant}

\author[C. Sung]{Chanyoung Sung}
 \address{Dept. of Mathematics Education  \\
          Korea National University of Education
          }
\email{cysung@kias.re.kr}




\date{}

\begin{abstract}
It is shown that on every closed oriented Riemannian 4-manifold
$(M,g)$ with positive scalar curvature,  $$\int_M|W^+_g|^2d\mu_{g}\geq 2\pi^2(2\chi(M)+3\tau(M))-\frac{8\pi^2}{|\pi_1(M)|},$$ where $W^+_g$, $\chi(M)$ and $\tau(M)$ respectively denote the self-dual Weyl tensor of $g$, the Euler characteristic and the signature of $M$. This generalizes Gursky's inequality \cite{gur} for the case of $b_1(M)>0$ in a much simpler way.

We also extend all such lower bounds of the Weyl functional to 4-orbifolds including Gursky's inequalities for the case of $b_2^+(M)>0$ or $\delta_gW^+_g=0$, and obtain topological obstructions to the existence of self-dual orbifold metrics of positive scalar curvature.
\end{abstract}

\maketitle


\tableofcontents

\setcounter{section}{0}
\setcounter{equation}{0}

\section{Introduction and statement of main results}


In a Riemannian manifold $(M^n,g)$ its curvature tensor is decomposed into its irreducible summands under the orthogonal group, i.e
the scalar curvature $s_g$, the trace-free part $\stackrel{\circ}r_{g}:=\textrm{Ric}_g-\frac{s_g}{n}g$ of the Ricci curvature $\textrm{Ric}_g$,
and the Weyl curvature $W_g$. Integrals of certain combinations of these curvature parts give not only topological invariants of $M$, but an optimal metric on $M$ can also be often found by minimizing the $L^{\frac{n}{2}}$-norm of a part of the curvature tensor. The most well-studied example is the case of the scalar curvature, known as the Yamabe problem \cite{L&P}.

From now on $M$ is supposed as smooth and closed.
The \textit{Yamabe constant} $Y(M,[g])$ of a conformal class $$[g]:= \{\psi g \mid \psi:M \stackrel{C^\infty}{\rightarrow} \Bbb R_+ \}$$ is defined as  $\inf_{\hat{g}\in [g]}\mathfrak{F}(\hat{g})$ for $$\mathfrak{F}(\hat{g}):=\frac{\int_M s_{\hat{g}}\
d\mu_{\hat{g}}}{\textrm{Vol}(\hat{g})^{\frac{n-2}{n}}}$$ where $\textrm{Vol}(\hat{g})$ is the volume of $(M,\hat{g})$. The infimum is always attained by a smooth metric called \textit{Yamabe metric} with constant scalar curvature. Moreover  $Y(M,[g])$ can be also expressed as
$$\inf\{Y_g (f)|f\in L^2_1 (M)-\{0\}\}$$ where the Yamabe functional $Y_g$ is defined as
$$Y_g (f) := \frac{\int_M (a_n|df|_g^2 + s_g f^2) \ d\mu_g }{ {\left( \int_M |f|^{p_n}\ d\mu_g \right) }^{\frac{2}{p_n}}}$$
for $a_n = \frac{4(n-1)}{n-2}$ and $p_n=\frac{2n}{n-2}$. The \textit{Yamabe invariant} $Y(M)$ is defined by
the supremum of $Y(M,[g])$ over all conformal classes of $M$.
An important formula derived from the solution of the Yamabe problem and the H\"older inequality is that
\begin{eqnarray}\label{inflation}
|Y(M,[g])|=(\inf_{\hat{g}\in[g]}\int_M|s_{\hat{g}}|^{\frac{n}{2}}d\mu_{\hat{g}})^{\frac{2}{n}}
\end{eqnarray}
where the infimum is realized by Yamabe metrics and its proof is given in \cite{LeB1}.

In this paper we are concerned with the Weyl functional $\mathcal{W}$ defined as
$$\mathcal{W}(M,[g]):=\int_M|W_g|^{\frac{n}{2}}d\mu_{g}$$ which is an invariant of a conformal class $[g]$. By $|W_g|$ we always  mean the norm $(\frac{1}{4}W_{ijkl}W^{ijkl})^{\frac{1}{2}}$ of $W_g\in \textrm{End}(\wedge^2(M))$ w.r.t. the metric $g$ under consideration.
O. Kobayashi \cite{koba} defined a topological invariant $$\nu(M):=\inf_g \mathcal{W}(M,[g])$$ where $g$ runs over all smooth Riemannian metrics on $M$. Unless there is a conformally-flat metric or a collapsing sequence of metrics with bounded curvature, it's very difficult to compute this invariant of a general manifold.

When the dimension is 4, the Weyl curvature acquires not only physical meanings \cite{penrose, witten}, but also an additional structure.
If $(M^4,g)$ is oriented, $W_g$ is further reducible under the special orthogonal group into $W_g^++W_g^-$.
The curvature integrals $\mathcal{W}_{\pm}(M,[g]):=\int_M|W^{\pm}_g|^2d\mu_{g}$ are related by
\begin{eqnarray*}
\mathcal{W}(M,[g])&=& \mathcal{W}_+(M,[g])+\mathcal{W}_-(M,[g]) \\ &=& -12\pi^2\tau(M)+2\mathcal{W}_+(M,[g])\\ &=& 12\pi^2\tau(M)+2\mathcal{W}_-(M,[g])
\end{eqnarray*}
where $\tau(M)$ denotes the signature of $M$.
So any estimation of $\mathcal{W}_+$ is equivalent to that of $\mathcal{W}$, and
\begin{eqnarray}\label{DSLee}
\nu(M)\geq 12\pi^2\tau(M)
\end{eqnarray}
with equality being held if $M$ admits a self-dual metric, i.e $\mathcal{W}_-(M,[g])=0$.
Thus the $\nu$ invariant of a connected sum of self-dual manifolds is computable by using the connected sum formula \cite{koba} $$\nu(M_1\# M_2)\leq \nu(M_1)+\nu(M_2),$$ and this will be discussed in Subsection \ref{coffee2u}.

It turns out that the estimation of $\mathcal{W}(M,[g])$ on $(M^4,[g])$ with positive Yamabe constant gets more accessible, as the following theorems show. Thus as a means to quantify the nonexistence of a self-dual metric with positive scalar curvature
and for the purpose of finding a canonical one among all metrics with positive scalar curvature, we are lead to define a refined invariant $$\nu_+(M):=\inf_{g\in \mathcal{C}_+} \mathcal{W}(M,[g])$$ where $\mathcal{C}_+$ is the set of all smooth Riemannian metrics on $M$ satisfying $Y(M,[g])>0$. (If $M$ has $\mathcal{C}_+=\emptyset$, $\nu_+(M)$ may be defined to be $\infty$.)


\begin{Theorem}[Chang-Gursky-Yang \cite{CGY2, CGY}]
Let $(M,g)$ be a smooth closed Riemannian 4-manifold with $Y(M,[g])>0$, which is not diffeomorphic to $S^4$ or $\Bbb RP^4$.
Then
\begin{eqnarray}\label{thank}
\mathcal{W}(M,[g])\geq 4\pi^2\chi(M)
\end{eqnarray}
where the equality holds iff $(M,g)$ is conformal to $\Bbb CP^2$ with the Fubini-Study metric or a quotient of $S^1\times S^3$ with a product metric of standard metrics. If
$$
\mathcal{W}(M,[g])< 8\pi^2\chi(M),
$$
then there exists a metric in $[g]$ with positive Ricci curvature.
\end{Theorem}

Fu \cite{Fu} and Bour and Carron \cite{bour} obtained similar types of rigidity results under upper bounds of $\mathcal{W}$.
A series of Gursky's remarkable results asserts that the existence of a nontrivial harmonic tensor requires the increase of $\mathcal{W}$ :
\begin{Theorem}[Gursky \cite{gur, gur2}]\label{th1}
Let $(M,g)$ be a smooth closed oriented Riemannian 4-manifold with $Y(M,[g])> 0$.
\begin{description}
  \item[(i)] If $b_2^+(M)>0$, then  $\mathcal{W}_+(M,[g])\geq \frac{4\pi^2}{3}(2\chi(M)+3\tau(M)),$
   where the equality holds iff $g$ is conformal to a K\"ahler-Einstein metric.
  \item[(ii)] If $b_1(M)>0$, then $\mathcal{W}_+(M,[g])\geq 2\pi^2(2\chi(M)+3\tau(M))$, where the equality holds iff $g$ is conformal to a quotient of $S^3\times \Bbb R$.
  \item[(iii)] If $\delta_gW_g^+=0$ for nonzero $W_g^+$, then $\mathcal{W}_+(M,[g])\geq \frac{4\pi^2}{3}(2\chi(M)+3\tau(M))$, where the equality holds iff $(M,g)$ is Einstein which is either K\"ahler or the quotient of a K\"ahler manifold by an anti-holomorphic isometric involution.
\end{description}
\end{Theorem}

The 1st purpose of this article is to generalize the part (ii) of Gursky's theorem to any closed oriented 4-manifolds by using a simpler method than that of \cite{gur}.
\begin{Theorem}\label{th2}
Let $(M,g)$ be a smooth closed oriented Riemannian 4-manifold with $Y(M,[g])\geq 0$, and $|\cdot|$ of a group denote its order.
\begin{description}
  \item[(i)] If $|\pi_1(M)|< \infty$, then $$\mathcal{W}_+(M,[g])\geq 2\pi^2(2\chi(M)+3\tau(M))-\frac{8\pi^2}{|\pi_1(M)|},$$ where the equality holds iff $(M,g)$ is conformal to a round 4-sphere.
  \item[(ii)] If  $|H_1(M;\Bbb Z)|<\infty$, then  $$\mathcal{W}_+(M,[g])\geq 2\pi^2(2\chi(M)+3\tau(M))-\frac{8\pi^2}{|H_1(M;\Bbb Z)|},$$  where the equality holds iff $(M,g)$ is conformal to a round 4-sphere.
  \item[(iii)] If $\pi_1(M)$ contains a subgroup of arbitrarily large finite index, then   $$\mathcal{W}_+(M,[g])\geq 2\pi^2(2\chi(M)+3\tau(M)).$$
\end{description}
\end{Theorem}

One can carry all the above discussion more generally to the orbifold setting. An $n$-orbifold is a Hausdorff 2nd-countable topological space which is locally modelled on a quotient $\Bbb R^n/\Gamma$ of $\Bbb R^n$ by a finite group action $\Gamma$. In this paper we assume that every orbifold is smooth by requiring that all transition functions are smooth, and $\Gamma$ is a subgroup of $SO(n)$ acting freely on $\Bbb R^n-\{0\}$. An orientation of an orbifold is meant by that on the complement of its singular locus, and we regard a manifold as a special kind of orbifolds.
A smooth Riemannian $n$-orbifold is a smooth Riemannian manifold outside orbifold-singular points around each of which the metric is locally the projection of a smooth $\Gamma$-invariant metric on $\Bbb R^n$. By an orbifold metric we always mean a smooth Riemannian orbifold metric.

Orbifolds naturally arise as a limit or quotients of manifolds with a geometric structure, and they are being studied extensively not only in geometry and but also in physics such as string theory. Many classical theorems on manifolds are generalized to orbifolds. For instance the DeRham theorem, the Hodge theorem, and the Chern-Gauss-Bonnet theorem hold for a smooth closed orbifold as well \cite{Baily, Satake, sat}.
Also the Yamabe problem on an orbifold can be defined in the obvious way, and turns out to be almost parallel to the case of a smooth manifold. For details, the reader may consult K. Akutagawa  \cite{aku}.

A Yamabe metric of a smooth closed Riemannian $n$-orbifold $(M,g)$ exists if
\begin{eqnarray}\label{Ahncheolsoo}
Y(M,[g])< \min_{i}\frac{Y(S^n)}{|\Gamma_i|^{\frac{2}{n}}},
\end{eqnarray}
where the minimum is taken over all the orbifold groups $\Gamma_i$, and the non-strict inequality called generalized Aubin's inequality always holds. The possible non-solvability of the orbifold Yamabe problem when Aubin's inequality is saturated is the only essential difference from the ordinary Yamabe problem on manifolds. As long as the Yamabe problem is solvable, the formula (\ref{inflation}) holds for orbifolds as well.

Using the following topological invariants to be explained later
$$\chi_{orb}(M)=\chi(M)-\sum_i(1-\frac{1}{|\Gamma_i|}),\ \ \ \ \ \tau_{orb}(M)=\tau(M)+\sum_i\eta(S^3/\Gamma_i),$$ 
and $\pi_1^{orb}(M)$ which is the deck transformation group of the so-called \textit{universal orbifold covering space},
the inequalities of the above theorems are extended to a 4-orbifold as follows :
\begin{Theorem}\label{th3}
Let $(M,g)$ be a smooth closed oriented Riemannian 4-orbifold with orbifold groups $\Gamma_0,\cdots,\Gamma_m$. Suppose that $Y(M,[g])> 0$.
\begin{description}
  \item[(i)] If $|\pi_1^{orb}(M)|< \infty$, then $$\mathcal{W}_+(M,[g])\geq 2\pi^2(2\chi_{orb}(M)+3\tau_{orb}(M))-\min_{j}\frac{8\pi^2}{|\pi_1^{orb}(M)||\tilde{\Gamma}_j|}$$ where $\tilde{\Gamma}_j$ are orbifold groups of the universal orbifold covering space of $M$.
  \item[(ii)] If $|\pi_1(M)|< \infty$, then $$\mathcal{W}_+(M,[g])\geq 2\pi^2(2\chi_{orb}(M)+3\tau_{orb}(M))-\min_{j}\frac{8\pi^2}{|\pi_1(M)||\Gamma_j|}.$$
  \item[(iii)] If  $|H_1(M;\Bbb Z)|<\infty$, then  $$\mathcal{W}_+(M,[g])\geq 2\pi^2(2\chi_{orb}(M)+3\tau_{orb}(M))-\min_{j}\frac{8\pi^2}{|H_1(M;\Bbb Z)||\Gamma_j|}.$$
  \item[(iv)] If $\pi_1(M)$ or $\pi_1^{orb}(M)$ contains a subgroup of arbitrarily large finite index, then  $$\mathcal{W}_+(M,[g])\geq 2\pi^2(2\chi_{orb}(M)+3\tau_{orb}(M)).$$
  \item[(v)] If $b_2^+(M)>0$, then $$\mathcal{W}_+(M,[g])\geq \frac{4\pi^2}{3}(2\chi_{orb}(M)+3\tau_{orb}(M))$$ where the equality  holds iff $g$ is conformal to an orbifold  K\"ahler-Einstein metric.
\item[(vi)] If $\delta_gW_g^+=0$ for nonzero $W_g^+$, then $$\mathcal{W}_+(M,[g])\geq \frac{4\pi^2}{3}(2\chi_{orb}(M)+3\tau_{orb}(M))$$ where the equality holds iff $(M,g)$ is Einstein which is either K\"ahler or the quotient of a K\"ahler orbifold by a free anti-holomorphic isometric involution.
\end{description}
\end{Theorem}

We shall provide examples where our inequalities give sharper lower bounds of $\mathcal{W}_+$ and discuss some applications such as nonexistence of self-dual metrics of positive scalar curvature.

\section{Proof of Theorem \ref{th2}}
Our proof is based on the well-known formula
\begin{eqnarray}\label{Choomiae}
2\pi^2 (2\chi(M)+3\tau(M))=\int_M (|W^+_{g}|^2+\frac{s_{g}^2}{48}-\frac{|\stackrel{\circ}r_{g}|^2}{4})\
d\mu_{g}
\end{eqnarray}
which is the combination of the Chern-Gauss-Bonnet formula and the Hirzebruch signature formula.
\begin{Lemma}\label{Carol}
Let $(M,g)$ be a smooth closed oriented Riemannian 4-manifold with a $k$-fold covering $p:\tilde{M}\rightarrow M$. If $Y(M,[g])\geq 0$, then
$$2\pi^2 (2\chi(M)+3\tau(M))-
\int_M |W^+_{g}|^2d\mu_{g}\leq \frac{8\pi^2}{k},$$ where the equality holds iff $k=1$ and $(M,g)$ is conformal to a round 4-sphere.
\end{Lemma}
\begin{proof}
Put
\begin{eqnarray*}
d&:=&2\pi^2 (2\chi(M)+3\tau(M))-
\int_M |W^+_{g}|^2d\mu_{g}.
\end{eqnarray*}
By (\ref{Choomiae})
$$d=\int_M
(\frac{s_{g}^2}{48}-\frac{|\stackrel{\circ}r_{g}|^2}{4})\
d\mu_{g},$$ and hence
$$ \int_{\tilde{M}}
(\frac{s_{p^*g}^2}{48}-\frac{|\stackrel{\circ}r_{p^*g}|^2}{4})\
d\mu_{p^*g}=kd.$$ For any metric $h\in [p^*g]$ and the pullback orientation (or its reversed one) on $\tilde{M}$,
\begin{eqnarray*}
\int_{\tilde{M}}
(\frac{s_{h}^2}{48}-\frac{|\stackrel{\circ}r_{h}|^2}{4})\
d\mu_{h}&=& 2\pi^2 (2\chi(\tilde{M})+3\tau(\tilde{M}))-
\int_{\tilde{M}} |W^+_{h}|^2d\mu_{h}\\  &=& 2\pi^2 (2\chi(\tilde{M})+3\tau(\tilde{M}))-
\int_{\tilde{M}} |W^+_{p^*g}|^2d\mu_{p^*g}\\ &=& \int_{\tilde{M}}
(\frac{s_{p^*g}^2}{48}-\frac{|\stackrel{\circ}r_{p^*g}|^2}{4})\
d\mu_{p^*g}\\ &=&  kd.
\end{eqnarray*}
Thus if $d\geq 0$, then
 $$\inf_{h\in [p^*g]} (\int_{\tilde{M}}
s_{h}^2d\mu_{h})^{\frac{1}{2}}\geq 4\sqrt{3kd}.$$
Since $Y(\tilde{M},[p^*g])\geq 0$, $$Y(\tilde{M},[p^*g])=\inf_{h\in [p^*g]} (\int_{\tilde{M}}
s_{h}^2d\mu_{h})^{\frac{1}{2}}$$ by the formula (\ref{inflation}).

Therefore we conclude that $$d\leq \frac{Y(\tilde{M},[p^*g])^2}{48k}\leq  \frac{8\pi^2}{k},$$
where we used Aubin's inequality $Y(\tilde{M},[p^*g])\leq Y(S^4)=8\sqrt{6}\pi$.

The equality holds iff a Yamabe metric of $(\tilde{M},[p^*g])$ is Einstein with Yamabe constant equal to $Y(S^4)$. Recall that the only conformal class with its Yamabe constant equal to $Y(S^4)$ is that of a round metric on $S^4$. Thus the equality holds iff $k$ is equal to $|\pi_1(M)|$, and $(\tilde{M},p^*g)$ is conformal to $S^4$ with a round metric. By the Lefschetz fixed point theorem, any orientation-preserving diffeomorphism of $S^4$ has at least 2 fixed points. Therefore the only orientable manifold covered by $S^4$ is $S^4$ itself.
\end{proof}

Now if $|\pi_1(M)|< \infty$, then one can apply the above lemma with  $k=|\pi_1(M)|$ and $\tilde{M}$ equal to the universal covering space to get the desired inequality. Then the equality is attained iff $(M,g)$ is conformal to a round 4-sphere.

It remains to deal with $|\pi_1(M)|= \infty$ cases. If there exists a subgroup of $\pi_1(M)$ with index $k$, then there exists a $k$-fold covering space of $M$ and the above lemma applies.

To find a subgroup of finite index, we try to find a surjective homomorphism from $\pi_1(M)$ onto a finite group $G$.
The first obvious try is  the obvious quotient homomorphism  $\psi: \pi_1(M)\rightarrow H_1(M;\Bbb Z)$.
If $b_1(M)=0$, then $\ker\psi$ gives a (normal) subgroup of index $|H_1(M;\Bbb Z)|$. If $b_1(M)>0$, then we take a surjective homomorphism $\tilde{\psi}_k: H_1(M;\Bbb Z)\rightarrow \Bbb Z_k$ to get  a (normal) subgroup  $\ker(\tilde{\psi}_k\circ\psi)$  of index $k$ for any integer $k>0$. In general, if there exists a subgroup of $\pi_1(M)$ with index bigger than $m$ for any integer $m$, one can achieve $$2\pi^2 (2\chi(M)+3\tau(M))-\int_M |W^+_{g}|^2d\mu_{g}\leq 0.$$ This completes the proof.



\begin{Remark}
Here the condition of (iii) is satisfied, when $b_1(M)>0$, or more generally $\pi_1(M)$ is infinite but residually finite, i.e. for each non-identity element in the group, there is a normal subgroup of finite index not containing that element. By Malcev's theorem \cite{mal}, if $\pi_1(M)$ has a faithful representation on a finite dimensional vector space over a field, it is residually finite.

From the above proof, it's obvious that we can obtain a lower bound of $\mathcal{W}_+(M,[g])$ better than that of the inequality (i), if $|\pi_1(M)|<\infty$ and $Y(\tilde{M})$ of the universal cover $\tilde{M}$ of $M$ is strictly smaller than $Y(S^4)$.
\end{Remark}

\begin{Remark}
If $Y(M^4,[g])=0$, then from (\ref{Choomiae}) we still have
\begin{eqnarray}\label{ggg}
\mathcal{W}_+(M,[g])\geq 2\pi^2(2\chi(M)+3\tau(M)),
\end{eqnarray}
where the equality holds iff $g$ is conformal to a Ricci-flat metric.
\end{Remark}

\section{Proof of Theorem \ref{th3}}

Unlike the manifold case, there is a complication in the orbifold case, due to the possible non-solvability of the Yamabe problem. To bypass this difficulty, we need to approximate the given metric $g$ by a metric in which the Yamabe problem is solvable. This will constitute a large portion of this section.
In this paper $g_{_{\Bbb E^n}}$ and $g_{_{\Bbb S^n}}$ respectively denote the Euclidean metric on $\Bbb R^n$ and the standard round metric of constant curvature 1 on $S^n$ and its quotient.
Before entering into the proof, we need to review some basic facts in the Yamabe problem and orbifolds.

\subsection{Convergence of minimizing sequence}

A sequence $\{\varphi_i\in L_1^2(M)|i\in \Bbb N\}$ is called a minimizing sequence for $Y_g$ if $$\lim_{i\rightarrow \infty}Y_g(\varphi_i)= Y(M,[g]),$$ and  $\varphi\in L_1^2(M)$ is also called a Yamabe minimizer for $Y_g$ if $Y_g(\varphi)=Y(M,[g])$.
\begin{Theorem}\label{yam}
On a smooth closed Riemannian $n$-orbifold $(M,g)$ for $n\geq 3$, let $\{\varphi_i\in L_1^2(M)|i\in \Bbb N\}$ be a minimizing sequence for $Y_g$ such that $\lim_{i\rightarrow \infty}||\varphi_i||_{_{L^{^{p_n}}}}$ exists and is nonzero. If
there exists $\tilde{\varphi}\in L^{p_n}(M)$ such that $|\varphi_i|\leq \tilde{\varphi}$ for all $i$, then there exists a Yamabe minimizer $\varphi$ to which a subsequence of $\{\varphi_i\}$ converges in $L_1^2$-norm.
\end{Theorem}
\begin{proof}
Set $a:=a_n$ and $p:=p_n$.
We may assume that $$\int_M|\varphi_i|^{p}d\mu_g=1$$ for all $i\in \Bbb N$ by considering $\frac{\varphi_i}{||\varphi_i||_{_{L^{p}}}}$.

From
\begin{eqnarray}\label{jongguk-0}
\lim_{i\rightarrow \infty}\int_M(a|d\varphi_i|^2+s_g\varphi_i^2)\ d\mu_g&=&Y(M,[g]),
\end{eqnarray}
there exists an integer $n_0$ such that if $i\geq n_0$ then
\begin{eqnarray*}
Y(M,[g])+1 &\geq& \int_M(a|d\varphi_i|^2+s_g\varphi_i^2)\ d\mu_g\\ &\geq& \int_M(a|d\varphi_i|^2-|\min_Ms_g|\ \varphi_i^2)\ d\mu_g\\ &\geq& \int_Ma|d\varphi_i|^2 d\mu_g-C(\int_M|\varphi_i|^p\ d\mu_g)^{\frac{2}{p}}
\end{eqnarray*}
for a constant $C>0$, implying that $$\sup_i\int_M|d\varphi_i|^2d\mu_g\leq C'$$ for a constant $C'>0$, and hence $\{\varphi_i|i\in \Bbb N\}$ is a bounded subset of $L_1^2(M)$.

Thus there exists $\varphi\in L_1^2(M)$ and a subsequence converging to $\varphi$ weakly in $L_1^2$, strongly in $L^2$, and pointwisely almost everywhere.\footnote{For the strong convergence in $L^2$, we used the Rellich-Kondrakov theorem which holds still on orbifolds. It can be easily derived by using the partition of unity. For a proof one may refer to \cite{Falsi}.} By abuse of notation we let $\{\varphi_i\}$ be the subsequence.

Owing to that $|\varphi_i|\leq \tilde{\varphi}\in L^p(M)$, we can apply Lebesgue's dominated convergence theorem to obtain
\begin{eqnarray}\label{jongguk-1}
\int_Ms_g\varphi_i^2\ d\mu_g\rightarrow \int_Ms_g\varphi^2d\mu_g
\end{eqnarray}
 and
\begin{eqnarray}\label{jongguk-2}
\int_M|\varphi_i|^p d\mu_g\rightarrow \int_M|\varphi|^pd\mu_g,
\end{eqnarray}
 and hence $\int_M|\varphi|^pd\mu_g=1$ and there must exist $\lim_{i\rightarrow \infty}\int_M|d\varphi_i|^2d\mu_g$ by (\ref{jongguk-0}).

By the weak convergence $\varphi_i\rightarrow\varphi$ in $L_1^2$,
$$\int_M|d\varphi|^2d\mu_g+\int_M\varphi^2\ d\mu_g = \lim_{i\rightarrow \infty}(\int_M\langle d\varphi_i,d\varphi\rangle d\mu_g+\int_M \varphi_i\varphi\  d\mu_g),$$ so
\begin{eqnarray}\label{jongguk-3}
\int_M|d\varphi|^2d\mu_g &=& \lim_{i\rightarrow \infty}\int_M\langle d\varphi_i,d\varphi\rangle d\mu_g\\ &\leq& \limsup_{i\rightarrow \infty}(\int_M|d\varphi_i|^2d\mu_g    )^{\frac{1}{2}}(\int_M|d\varphi|^2d\mu_g)^{\frac{1}{2}}\nonumber\\ &=& \lim_{i\rightarrow \infty}(\int_M|d\varphi_i|^2d\mu_g    )^{\frac{1}{2}}(\int_M|d\varphi|^2d\mu_g)^{\frac{1}{2}}\nonumber
\end{eqnarray}
implying that
$$\int_M|d\varphi|^2d\mu_g \leq \lim_{i\rightarrow \infty}\int_M|d\varphi_i|^2d\mu_g.$$

Therefore
$$Y(M,[g])=\lim_{i\rightarrow \infty}\frac{\int_M(a|d\varphi_i|^2+s_g\varphi_i^2)\ d\mu_g}{(\int_M|\varphi_i|^p\ d\mu_g)^{\frac{2}{p}}}\geq
\frac{\int_M(a|d\varphi|^2+s_g\varphi^2)\ d\mu_g}{(\int_M|\varphi|^p\ d\mu_g)^{\frac{2}{p}}},$$ and hence $\varphi$ must be a Yamabe minimizer.

By (\ref{jongguk-1}) and (\ref{jongguk-2}), this implies that $$\lim_{i\rightarrow \infty}\int_M|d\varphi_i|^2d\mu_g=\int_M|d\varphi|^2d\mu_g.$$
Combined with (\ref{jongguk-3}), it gives
\begin{eqnarray*}
\lim_{i\rightarrow \infty}\int_M|d\varphi-d\varphi_i|^2d\mu_g&=&\lim_{i\rightarrow \infty}\int_M(|d\varphi|^2-2\langle d\varphi_i,d\varphi\rangle+|d\varphi_i|^2)\ d\mu_g=0,
\end{eqnarray*}
completing the proof that $\varphi_i\rightarrow\varphi$ strongly in $L_1^2$.

\end{proof}

\subsection{Equivariant Yamabe problem}

Suppose a compact Lie group $G$ acts on a smooth closed Riemannian $n$-manifold $(X,g)$ for $n\geq 3$ smoothly and isometrically. We define $[g]_G$ to be the set of all smooth $G$-invariant metrics conformal to $g$, and $$Y(X,[g]_G):=\inf_{\hat{g}\in [g]_G}\frac{\int_X s_{\hat{g}}\
d\mu_{\hat{g}}}{(\int_X d\mu_{\hat{g}})^{\frac{n-2}{n}}}.$$

According to Hebey and Vaugon \cite{HV},
$$Y(X,[g]_G)\leq Y(S^n) (\inf_{x\in M}|Gx|)^{\frac{2}{n}}$$ where $|Gx|$ denotes the cardinality of the orbit of $x$, and if the strict inequality holds, the equivariant Yamabe problem is solvable, i.e. $Y(X,[g]_G)$ is always achieved by a $G$-invariant metric of constant scalar curvature in $[g]_G$. They also conjectured that the strict inequality holds if $(X,g)$ is not conformal to $(S^n,g_{_{\Bbb S^n}})$ or if the $G$-action has no fixed point. Madani \cite{madani} resolved the conjecture for $n\leq 37$.

In case that $G$ is finite and the quotient space $X/G$ is an orbifold, the equivariant Yamabe problem on $X$ is equivalent to the orbifold Yamabe problem on $(X/G, \hat{g})$ where $\hat{g}$ is the orbifold metric induced by the $G$-invariant metric $g$ and $Y(X/G,[\hat{g}])$ is equal to $\frac{Y(X,[g]_G)}{|G|^{\frac{2}{n}}}$.

The interested reader might consult \cite{sung} for some nontrivial computations of equivariant Yamabe invariants.

\subsection{Topology of orbifolds}
For a detailed explanation on a general $n$-orbifold, the reader might consult \cite{Thurston, SChoi}.

A local (uniformizing) chart of an $n$-orbifold is given by $U\subset \Bbb R^n$ modulo a finite group $\Gamma < SO(n)$, and a smooth orbifold map $f:U_1/\Gamma_1\rightarrow U_2/\Gamma_2$ seen on local charts is a continuous map with a smooth lifting $\tilde{f}: U_1\rightarrow U_2$ and a homomorphism $\gamma:\Gamma_1\rightarrow \Gamma_2$ such that $$\tilde{f}(g\cdot x)=\gamma(g)\cdot \tilde{f}(x)$$ for all $g\in \Gamma_1$ and $x\in U_1$. For example of a smooth orbifold map, a smooth real-valued function on $U/\Gamma$ is a $\Gamma$-invariant smooth real-valued function on $U$.
We always assume that any orbifold map is smooth and an orbifold diffeomorphism is a smooth orbfold map whose inverse is also a smooth orbifold map.

As another important example of an orbifold map, an orbifold covering $\pi:\check{M}\rightarrow M$ between two orbifolds of the same dimension is an ordinary covering over the complement of the orbifold points and each orbifold point $x\in M$ has a neighborhood $N(x)$ with local uniformizing chart $U/\Gamma$ such that each component $\mathcal{U}_i$ of $\pi^{-1}(N(x))$ is diffeomorphic to $U/\check{\Gamma}$ for a subgroup $\check{\Gamma}\leq \Gamma$ and $\pi|_{\mathcal{U}_i}$ is the natural projection $U/\check{\Gamma}\rightarrow U/\Gamma$. Let us call such neighborhood $N(x)$ an \textit{evenly covered neighborhood} of $x$, and an orbifold covering is called $k$-fold or $k$-sheeted if a generic fiber consists of $k$ points.

From the definition an ordinary covering of an oribifold is also a special kind of orbifold coverings.
Two covering orbifolds $(\check{M}_1,\pi_1)$ and $(\check{M}_2,\pi_2)$ of $M$ are isomorphic if there is a fiber-preserving orbifold diffeomorphism. A covering automorphism or deck transformation is a covering isomorphism of itself.

Every orbifold $M$ has an universal orbifold covering $\pi:\tilde{M}\rightarrow M$ in the sense that for any orbifold covering $p:\check{M} \rightarrow M$ there is $\pi':\tilde{M}\rightarrow \check{M}$ such that $\pi=p\circ\pi'$. It is unique up to covering isomorphism.
The orbifold fundamental group $\pi_1^{orb}(M)$ of an orbifold is then defined as the group of deck transformations of the universal orbifold cover $\tilde{M}$, and
every orbifold covering space of $M$ is isomorphic to $\tilde{M}/\Gamma\rightarrow M$ for a subgroup $\Gamma$ of $\pi_1^{orb}(M)$. In fact, the isomorphism classes of orbifold covering spaces of $M$ are in 1-1 correspondence with the conjugacy classes of subgroups of $\pi_1^{orb}(M)$.

Now let's consider a smooth closed oriented Riemannian $4$-orbifold $(M,g)$. Its orbifold Euler characteristic and orbifold signature are defined as
$$\chi_{orb}(M):=\frac{1}{8\pi^2} \int_M (|W^+_{g}|^2+|W^-_{g}|^2+\frac{s_{g}^2}{24}-\frac{|\stackrel{\circ}r_{g}|^2}{2})\
d\mu_{g}$$ and
\begin{eqnarray}\label{knuefaculty}
\tau_{orb}(M):=\frac{1}{12\pi^2}(\mathcal{W}_+(M,[g])-\mathcal{W}_-(M,[g]))
\end{eqnarray}
respectively as in the Chern-Gauss-Bonnet formula and the Hirzebruch signature formula for manifolds.
As in manifolds, they are topological invariants given by
$$\chi_{orb}(M)=\chi(M)-\sum_i(1-\frac{1}{|\Gamma_i|}),\ \ \ \ \ \tau_{orb}(M)=\tau(M)+\sum_i\eta(S^3/\Gamma_i)$$ where the summation is over all the orbifold points $p_i$ and $\eta(S^3/\Gamma_i)$ is the eta invariant of $S^3/\Gamma_i$ with the standard metric.(\cite{Hit, sat, viaclo3}) The 1st equality states that each orbifold point $p_i$ contributes to $\chi_{orb}$ by $\frac{1}{|\Gamma_i|}$ rather than 1 and the 2nd equality is just a restatement of the Atiyah-Patodi-Singer index theorem for the signature operator.  Note that for a $k$-fold orbifold-covering $\check{M}$ of $M$,
\begin{eqnarray}\label{Gan-YH}
\chi_{orb}(\check{M})=k\chi_{orb}(M),\ \ \ \ \ \   \tau_{orb}(\check{M})=k\tau_{orb}(M),
\end{eqnarray}
because any smooth orbifold metric on $M$ can be lifted to a smooth orbifold metric on $\check{M}$.

In the same way as the manifold case (\ref{ggg}), we readily have that
if $Y(M,[g])=0$, then
\begin{eqnarray}\label{ggg-1}
\mathcal{W}_+(M,[g])\geq 2\pi^2(2\chi_{orb}(M)+3\tau_{orb}(M))
\end{eqnarray}
where the equality holds iff $g$ is conformal to a Ricci-flat orbifold metric.

\subsection{Comparison of Yamabe constant and Weyl functional}

\begin{Theorem}\label{th4}
Let $(M,g)$ be a smooth closed oriented Riemannian 4-orbifold of $b_2^+(M)\geq 1$. Then $$Y(M,[g])\leq 2\sqrt{6}(\mathcal{W}_+(M,[g]))^{\frac{1}{2}},$$
and the equality holds iff $[g]$ contains a K\"ahler metric of nonnegative constant scalar curvature which is a Yamabe metric.
\end{Theorem}
\begin{proof}
Let $\psi$ be a nonzero self-dual harmonic 2-form of $(M,g)$. By the Weitzenb\"ock formula \cite{bour} for a self-dual 2-form
$$0=(d+d^*)^2\psi=\nabla^*\nabla \psi -2W^+(\psi,\cdot)+\frac{s}{3}\psi$$
where and below $W^+$ and $s$ mean $W_g^+$ and $s_g$ respectively.
Taking an inner product with $\psi$ and using the identity
$|W^+||\psi|^2\geq \sqrt{\frac{3}{2}}W^+(\psi,\psi)$ gives
\begin{eqnarray*}
0&\geq& -\frac{1}{2}\Delta|\psi|^2+|\nabla\psi|^2-\frac{2}{3}\sqrt{6}|W^+||\psi|^2+\frac{s}{3}|\psi|^2.
\end{eqnarray*}

For any $\epsilon>0$, away from the zero locus of $\psi$
\begin{eqnarray*}
\frac{3|d|\psi|^2|^2}{8(|\psi|^2+\epsilon)}&\leq& \frac{3|d|\psi|^2|^2}{8|\psi|^2}\\ &=&\frac{3}{2}|d|\psi||^2\\ &\leq& |\nabla\psi|^2
\end{eqnarray*}
where the last line is a refined Kato inequality \cite{kato}.
At any point where $\psi=0$, a smooth function $|\psi|^2$ attains its minimum so that $d|\psi|^2=0$ there, implying that the above inequality
\begin{eqnarray}\label{cohort-2}
\frac{3|d|\psi|^2|^2}{8(|\psi|^2+\epsilon)}\leq |\nabla \psi|^2
\end{eqnarray}
holds true everywhere in $M$.

So we have that for any $\epsilon>0$
\begin{eqnarray*}
0&\geq& -\frac{3}{2}\Delta|\psi|^2+\frac{9|d|\psi|^2|^2}{8(|\psi|^2+\epsilon)}+(s-2\sqrt{6}|W^+|)|\psi|^2.
\end{eqnarray*}
at any point of $M$.

Multiply both sides by $(|\psi|^2+\epsilon)^{-\frac{1}{2}}$, and integrate by parts to get
$$0\geq \int_M(-\frac{3|d|\psi|^2|^2}{4(|\psi|^2+\epsilon)^{\frac{3}{2}}}+\frac{9|d|\psi|^2|^2}{8(|\psi|^2+\epsilon)^{\frac{3}{2}}}
+\frac{(s-2\sqrt{6}|W^+|)|\psi|^2}{(|\psi|^2+\epsilon)^{\frac{1}{2}}})\ d\mu_g$$
so that
\begin{eqnarray*}
0&\geq&\int_M(\frac{3|d|\psi|^2|^2}{8(|\psi|^2+\epsilon)^{\frac{3}{2}}}
+\frac{(s-2\sqrt{6}|W^+|)|\psi|^2}{(|\psi|^2+\epsilon)^{\frac{1}{2}}})\ d\mu_g\\
&=&   \int_M(6|d u_\epsilon|^2+(s-2\sqrt{6}|W^+|)(u_\epsilon^2-\epsilon u_\epsilon^{-2})\ d\mu_g\\
&\geq& \int_M(6|d u_\epsilon|^2+(s-2\sqrt{6}|W^+|)u_\epsilon^2-C_1\epsilon^{\frac{1}{2}})\ d\mu_g
\end{eqnarray*}
where we set $u_\epsilon:=(|\psi|^2+\epsilon)^{\frac{1}{4}}$, and $C_1>0$ is a constant.
Therefore
\begin{eqnarray*}
\int_M(6|d u_\epsilon|^2+su_\epsilon^2)\ d\mu_g &\leq& 2\sqrt{6}\int_M|W^+|u_\epsilon^2\ d\mu_g+ C_2\epsilon^{\frac{1}{2}}\\ &\leq&
2\sqrt{6}(\int_M|W^+|^2d\mu_g)^{\frac{1}{2}}(\int_Mu_\epsilon^4\ d\mu_g)^{\frac{1}{2}}+ C_2\epsilon^{\frac{1}{2}}
\end{eqnarray*}
for a constant $C_2>0$ by applying the H\"older inequality, and finally we have
\begin{eqnarray}\label{cohort-1}
Y_g(u_\epsilon)\leq 2\sqrt{6}(\int_M|W^+|^2d\mu_g)^{\frac{1}{2}} +\frac{C_2\epsilon^{\frac{1}{2}}}{||u_\epsilon||^2_{L^4}}.
\end{eqnarray}
Since $\epsilon>0$ is arbitrary and $$\lim_{\epsilon\rightarrow 0}||u_\epsilon||^2_{L^4}=(\int_M|\psi|^{2}d\mu_g)^{\frac{1}{2}}\ne 0$$ by Lebesgue's dominated convergence theorem, we can conclude that $$Y(M,[g])\leq \lim_{\epsilon\rightarrow 0}Y_g(u_\epsilon)\leq 2\sqrt{6}(\int_M|W^+|^2d\mu_g)^{\frac{1}{2}}.$$

To decide the equality case,  suppose $Y(M,[g])=2\sqrt{6}(\mathcal{W}_+(M,[g]))^{\frac{1}{2}}$.
Then as $\epsilon$ tends to 0, those $u_\epsilon$ for $0<\epsilon \ll 1$ give a minimizing sequence of $Y_g$ such that $\lim_{\epsilon\rightarrow 0}||u_\epsilon||_{L^4}$ is a nonzero constant and $|u_\epsilon|< (|\psi|^2+1)^{\frac{1}{4}}\in L^4(M)$.
By Theorem \ref{yam}, we can extract its subsequence converging to a Yamabe minimizer in $L_1^2$-norm. Thus the pointwise limit $\lim_{\epsilon\rightarrow 0}u_\epsilon=|\psi|^{\frac{1}{2}}$ is a minimizer of $Y_g$ and hence $|\psi|$ is nowhere vanishing by the well-known maximum principle of the Yamabe equation.

Now we choose a representative of $[g]$ such that $|\psi|$ is constant, and let's denote it still by $g$. Since $|\psi|$ is nowhere vanishing, one can proceed the above procedure with this $g$ and $\epsilon=0$, which leads to (\ref{cohort-1}) with $\epsilon=0$.
So the constant function $|\psi|^{\frac{1}{2}}=u_0$ is a minimizer of $Y_g$ implying that $g$ is a Yamabe metric, and the inequality (\ref{cohort-2}) with $\epsilon=0$ must be saturated. Therefore $\nabla \psi\equiv 0$ and the existence of a nontrivial parallel self-dual 2-form implies that $g$ is K\"ahler.

Conversely, suppose that $g$ is a K\"ahler Yamabe metric with nonnegative constant scalar curvature. Then taking $\psi$ to
be the K\"ahler form, by a well-known fact (\cite{Be}) that $$W^+(\psi)=\frac{s}{6}\psi,\ \ \  \textrm{and}\ \ \  W^+(\eta)=-\frac{s}{12}\eta$$ for
any self-dual 2-form $\eta$ pointwise-orthogonal to $\psi$, we have $$|W^+|=\frac{s}{2\sqrt{6}}$$
and
\begin{eqnarray*}
Y(M,[g])&=& \frac{\int_M s\ d\mu_g}{(\int_Md\mu_g)^{\frac{1}{2}}}\\ &=& (\int_Ms^2d\mu_g)^{\frac{1}{2}}\\ &=& (\int_M(2\sqrt{6}|W^+|)^2d\mu_g)^{\frac{1}{2}}.
\end{eqnarray*}
\end{proof}

The same inequality also holds when self-dual Weyl tensor is not identically zero and harmonic, i.e. $$\delta W^+=0.$$ It is known that $W^+$ is harmonic if Ricci tensor is parallel.
\begin{Theorem}\label{th5}
Let $(M,g)$ be a smooth closed oriented Riemannian 4-orbifold with nonzero harmonic self-dual Weyl tensor. Then $$Y(M,[g])\leq 2\sqrt{6}(\mathcal{W}_+(M,[g]))^{\frac{1}{2}}$$ where the equality holds iff $g$ is a Yamabe metric with positive constant scalar curvature and $W_g^+$ is parallel with exactly two eigenvalues at each point.
\end{Theorem}
\begin{proof}
We use the idea of Proposition 3.4 of \cite{gur2}, but we reduce the proof of \cite{gur2} much without going through the Yamabe problem for the so-called modified scalar curvature, and our proof is overall similar to the previous theorem. We abbreviate $W_g^+$ and $s_g$  by $W^+$ and $s$ respectively.

By the Weitzenb\"ock formula in \cite{Be, derdzinski} and using the identity
\begin{eqnarray}\label{Wuhanvirus}
\det W^+\leq \frac{\sqrt{6}}{18}|W^+|^3
\end{eqnarray}
whose equality is attained at a point where $W^+\ne 0$ iff $W^+$ at that point has precisely two eigenvalues,
\begin{eqnarray}\label{Wuhanvirus-2}
\Delta|W^+|^2&=&2|\nabla W^+|^2-36\det W^++s|W^+|^2\nonumber\\ &\geq& 2|\nabla W^+|^2-2\sqrt{6}|W^+|^3+s|W^+|^2.
\end{eqnarray}

For any $\epsilon>0$, away from the zero locus of $W^+$.
\begin{eqnarray*}
\frac{5|d|W^+|^2|^2}{12(|W^+|^2+\epsilon)}&\leq& \frac{5|d|W^+|^2|^2}{12|W^+|^2}\\ &=&\frac{5}{3}|d|W^+||^2\\ &\leq& |\nabla W^+|^2
\end{eqnarray*}
where the last line is a refined Kato inequality (see Lemma 2.1 in \cite{gur2} or Lemma 4 of \cite{GL}).
At any point where $W^+=0$, a smooth function $|W^+|^2$ attains its minimum so that $d|W^+|^2=0$ there, implying that the above inequality
\begin{eqnarray}\label{cohort-4}
\frac{5|d|W^+|^2|^2}{12(|W^+|^2+\epsilon)}\leq |\nabla W^+|^2
\end{eqnarray}
holds true everywhere in $M$.
So we have that
\begin{eqnarray*}
\Delta|W^+|^2&\geq& \frac{5|d|W^+|^2|^2}{6(|W^+|^2+\epsilon)}+(s-2\sqrt{6}|W^+|)|W^+|^2
\end{eqnarray*}
at any point of $M$.

Multiply both sides by $(|W^+|^2+\epsilon)^{-\frac{2}{3}}$, and integrate by parts to get
$$\int_M\frac{2|d|W^+|^2|^2}{3(|W^+|^2+\epsilon)^{\frac{5}{3}}}d\mu_g\geq\int_M(\frac{5|d|W^+|^2|^2}{6(|W^+|^2+\epsilon)^{\frac{5}{3}}}
+\frac{(s-2\sqrt{6}|W^+|)|W^+|^2}{(|W^+|^2+\epsilon)^{\frac{2}{3}}})\ d\mu_g$$
so that
\begin{eqnarray*}
0&\geq&\int_M(\frac{|d|W^+|^2|^2}{6(|W^+|^2+\epsilon)^{\frac{5}{3}}}
+\frac{(s-2\sqrt{6}|W^+|)|W^+|^2}{(|W^+|^2+\epsilon)^{\frac{2}{3}}})\ d\mu_g\\
&\geq&   \int_M(6|d u_\epsilon|^2+(s-2\sqrt{6}|W^+|)(u_\epsilon^2-\epsilon u_\epsilon^{-4})\ d\mu_g\\
&\geq& \int_M(6|d u_\epsilon|^2+(s-2\sqrt{6}|W^+|)u_\epsilon^2-C_1\epsilon^{\frac{1}{3}})\ d\mu_g
\end{eqnarray*}
where we set $u_\epsilon:=(|W^+|^2+\epsilon)^{\frac{1}{6}}$, and $C_1>0$ is a constant.
Therefore
\begin{eqnarray*}
\int_M(6|d u_\epsilon|^2+su_\epsilon^2)\ d\mu_g &\leq& 2\sqrt{6}\int_M|W^+|u_\epsilon^2\ d\mu_g+ C_2\epsilon^{\frac{1}{3}}\\ &\leq&
2\sqrt{6}(\int_M|W^+|^2d\mu_g)^{\frac{1}{2}}(\int_Mu_\epsilon^4\ d\mu_g)^{\frac{1}{2}}+ C_2\epsilon^{\frac{1}{3}}
\end{eqnarray*}
for a constant $C_2>0$ by applying the H\"older inequality, and finally we have
\begin{eqnarray}\label{cohort-3}
Y_g(u_\epsilon)\leq 2\sqrt{6}(\int_M|W^+|^2d\mu_g)^{\frac{1}{2}} +\frac{C_2\epsilon^{\frac{1}{3}}}{||u_\epsilon||^2_{L^4}}.
\end{eqnarray}
Since $\epsilon>0$ is arbitrary and $$\lim_{\epsilon\rightarrow 0}||u_\epsilon||^2_{L^4}=(\int_M|W^+|^{\frac{4}{3}}d\mu_g)^{\frac{1}{2}}\ne 0$$ by Lebesgue's dominated convergence theorem,  we can conclude that $$Y(M,[g])\leq \lim_{\epsilon\rightarrow 0}Y_g(u_\epsilon)\leq 2\sqrt{6}(\int_M|W^+|^2d\mu_g)^{\frac{1}{2}}.$$

To decide the equality case,  suppose $Y(M,[g])=2\sqrt{6}(\mathcal{W}_+(M,[g]))^{\frac{1}{2}}>0$.
Then in the same way as the previous theorem $|W^+|^{\frac{1}{3}}=\lim_{\epsilon\rightarrow 0}u_\epsilon$ must be a minimizer of $Y_g$ and hence $|W^+|$ is nowhere vanishing by the maximum principle. We choose a representative $g$ of $[g]$ such that $|W^+|$ is constant.
Since $|W^+|$ is nowhere vanishing, one can proceed the above procedure with this $g$ and $\epsilon=0$, which leads to (\ref{cohort-3}) with $\epsilon=0$.
So the constant function $|W^+|^{\frac{1}{3}}=u_0$ is a minimizer of $Y_g$ implying that $g$ is a Yamabe metric, and the inequality (\ref{cohort-4}) with $\epsilon=0$ must be saturated. Therefore $\nabla W^+\equiv 0$. Since the equality in (\ref{Wuhanvirus}) must be attained too,  $W^+$ has two eigenvalues at each point.


Conversely, suppose that $g$ is a Yamabe metric with positive constant scalar curvature and $W^+$ is parallel with exactly two eigenvalues at each point. Then equalities hold in (\ref{Wuhanvirus}) and (\ref{Wuhanvirus-2}), so $s\equiv 2\sqrt{6}|W^+|$.  Thus
\begin{eqnarray*}
Y(M,[g])&=& \frac{\int_M s\ d\mu_g}{(\int_Md\mu_g)^{\frac{1}{2}}}\\ &=& (\int_Ms^2d\mu_g)^{\frac{1}{2}}\\ &=& (\int_M(2\sqrt{6}|W^+|)^2d\mu_g)^{\frac{1}{2}}.
\end{eqnarray*}

\end{proof}

\subsection{Approximation lemma}

In this subsection, $B(r)\subset \Bbb R^4$ for any $r>0$ denotes the open ball of radius $r$ with center at the origin, and  following \cite{L&P} we write
$f=O''(|x|^k)$ for a smooth function $f:B(r)\rightarrow \Bbb R$ to mean $$f=O(|x|^k),\ \ \  \nabla f=O(|x|^{k-1}), \ \ \   \nabla^2 f=O(|x|^{k-2}),$$
i.e. there exists a constant $C>0$ such that $$|f(x)|\leq C|x|^k,\ \ \  |\nabla f(x)|\leq C|x|^{k-1}, \ \ \  |\nabla^2 f(x)|\leq C|x|^{k-2}$$ for all $x\in B(r)$.

A similar version of the following lemma is also proved in \cite{ABKS} by using a different method.
\begin{Lemma}\label{IKim}
Let $(M,g)$ be a smooth Riemannian 4-orbifold and $p\in M$ be any point. Then for any $\epsilon>0$ and neighborhood $V$ of $p$ there exists a smooth Riemannian orbifold metric $\bar{g}$ on $M$ such that $\bar{g}$ is flat in a neighborhood of $p$, it is equal to $g$ on $V^c$, and
\begin{eqnarray}\label{covid19}
|s_{\bar{g}}|<C,\  \ \ \ \ \ \ \ \ \ \ |\textrm{Vol}(\bar{g})-\textrm{Vol}(g)|<\epsilon
\end{eqnarray}
$$|Y(M,[g])-Y(M,[\bar{g}])|<\frac{\epsilon}{2}\ \ \ \ \textrm{and}\ \ \ \ |\mathcal{W}_+(M,[g])-\mathcal{W}_+(M,[\bar{g}])|<\frac{\epsilon}{2}$$ where $C>0$ is a constant independent of $\epsilon>0$ and $V$.
\end{Lemma}
\begin{proof}
Take a local uniformizing chart $U/\Gamma\subset M$ around $p$.
(If $p$ is not an orbifold point, then $\Gamma$ is the trivial group.)
By pull-back, $g$ gives a $\Gamma$-invariant smooth metric on $U$, which we still denote by $g$ by abuse of notation. Let's say that $B(d)$ for $d>0$ gives a normal coordinate of $U$ via the exponential map $\exp_g$ of $g$ at $p$.

We claim that $$g(x)=\sum_{i,j}(\delta_{ij}+O''(|x|^2))dx_idx_j$$ for $x\in B(d)$ by taking $d$ smaller if necessary.
Since all the 1st order derivatives of smooth functions $g_{ij}$ vanish at $0$, from their Taylor series expansions
\begin{eqnarray}\label{fareast-1}
g_{ij}(x)=\delta_{ij}+O(|x|^2)
\end{eqnarray}
\begin{eqnarray}\label{fareast-2}
\frac{\partial g_{ij}(x)}{\partial x_k}=O(|x|)
\end{eqnarray}
for all $i,j,k$ and $x\in B(d)$ by taking $d$ smaller if necessary.
Since $g_{ij}$ is smooth, it is obvious that
\begin{eqnarray}\label{fareast-3}
\frac{\partial^2 g_{ij}(x)}{\partial x_k^2}=O(1),
\end{eqnarray}
thereby justifying the claim.

From the above estimate we deduce that $$f_R^*(R^2g)=\sum_{i,j}(\delta_{ij}+\frac{1}{R^2}O''(|x|^2))dx_idx_j$$ for $x\in B(d)$
where  $$f_R:B(d)\rightarrow B(\frac{d}{R})$$ for $R\gg 1$ is the contraction map given by $x\mapsto \frac{x}{R}$.
To show it, let's denote $f_R^*(R^2g)$ by $\mathfrak{g}$ just for notational convenience.
From (\ref{fareast-1}), it readily follows that
\begin{eqnarray*}
\mathfrak{g}_{ij}(x)&=& R^2g(\frac{x}{R})(\frac{1}{R}\frac{\partial}{\partial x_i},\frac{1}{R}\frac{\partial}{\partial x_j})\\ &=& g_{ij}(\frac{x}{R})\\ &=& \delta_{ij}+O(|\frac{x}{R}|^2)\\ &=&\delta_{ij}+\frac{1}{R^2}O(|x|^2).
\end{eqnarray*}
By applying (\ref{fareast-2}) and the chain rule,
\begin{eqnarray*}
\frac{\partial \mathfrak{g}_{ij}(x)}{\partial x_k}&=&
\frac{\partial}{\partial x_k}g_{ij}(\frac{x}{R})\\ &=& \sum_m
\frac{\partial}{\partial (\frac{x_m}{R})}(g_{ij}(\frac{x}{R}))\cdot \frac{\partial}{\partial x_k}(\frac{x_m}{R})\\ &=&
\frac{\partial}{\partial (\frac{x_k}{R})}(g_{ij}(\frac{x}{R}))\cdot \frac{\partial}{\partial x_k}(\frac{x_k}{R})\\
&=& O(|\frac{x}{R}|)\frac{1}{R}\\ &=& \frac{1}{R^2}O(|x|)
\end{eqnarray*}
for all $i,j,k$.
Similarly applying (\ref{fareast-3}) and the chain rule twice,
\begin{eqnarray*}
\frac{\partial^2 \mathfrak{g}_{ij}(x)}{\partial x_k^2}&=& \frac{\partial^2}{\partial x_k^2}g_{ij}(\frac{x}{R})\\&=& \left(\frac{\partial}{\partial (\frac{x_k}{R})}\right)^2(g_{ij}(\frac{x}{R}))\cdot (\frac{\partial}{\partial x_k}(\frac{x_k}{R}))^2\\ &=& O(1)\frac{1}{R^2}
\end{eqnarray*}
for all $i,j,k$.

Let $\lambda : \Bbb R\rightarrow [0,1]$ be a fixed smooth decreasing function which is equal to 1 on $(-\infty,\frac{d}{3}]$ and 0 on $[\frac{2d}{3},\infty)$.
Define a $\Gamma$-invariant metric $$g_{_R}(x):=\lambda(|x|)g_{_{\Bbb E^4}}+(1-\lambda(|x|))f_R^*(R^2g)$$ on $B(d)$.
Since $g_{_R}$ at $\partial B(d)$ coincides with $f_R^*(R^2g)$, it extends to a metric on $M$, still denoted by $g_{_R}$  coinciding with $R^2g$ on the complement of $M_R:=\exp_g(B(\frac{d}{R}))/\Gamma$. For any sufficiently large $R$, $M_R$ is contained in $V$.

We claim that $\frac{1}{R^{2}}g_{_R}$ for any sufficiently large $R$ is the desired metric $\bar{g}$.
Since $f_R^*(R^2g)$ and hence $g_{_R}$ differ from $g_{_{\Bbb E^4}}$ on $B(d)$ by $\frac{1}{R^2}O''(|x|^2)$,
we may let
$$|R^2g-g_{_R}|_{R^2g}\leq \frac{C_3}{R^2},\ \ \ \ \
|s_{g_{_R}}-s_{_{R^2g}}|<\frac{C_3}{R^2},\ \ \ \ \ ||W^+_{g_{_R}}|-|W^+_{{R^2g}}||<\frac{C_3}{R^2}$$ on $M$ for a constant $C_3>0$ independent of $R>1$, where the pointwise norm of a $(0,2)$-tensor is measured w.r.t $R^2g$. So for $\bar{g}=\frac{1}{R^{2}}g_{_R}$
$$|g-\bar{g}|_g\leq \frac{C_3}{R^2},\ \ \ \ \
|s_{\bar{g}}-s_{g}|<C_3,\ \ \ \ \ ||W^+_{\bar{g}}|-|W^+_{g}||<C_3.$$ Since $\bar{g}=g$ outside of $M_R$,
the conditions of (\ref{covid19}) except the closeness of the Yamabe constants follow for any sufficiently large $R$.

To prove $$Y(M,[g_{_R}])< Y(M,[g])+\frac{\epsilon}{2}$$ for any sufficiently large $R$, we may write at each $x\in M$  $$d\mu_{g_{_R}}(x)=(1+\epsilon_1(x))d\mu_{_{R^2g}}(x),\ \ \ \ \ s_{g_{_R}}(x)=s_{_{R^2g}}(x)+\epsilon_2(x)$$
$$|\alpha|^2_{g_{_R}}\leq (1+\epsilon_3(x))|\alpha|^2_{_{R^2g}}$$ for any $\alpha\in T_x^*M$ such that all $\epsilon_i(x)$ are non-vanishing only for $x\in M_R$ and  $$\epsilon':=\max_{x\in M} (|\epsilon_1+(1+\epsilon_1)\epsilon_3|+|s_{_{R^2g}}\epsilon_1+(1+\epsilon_1)\epsilon_2|+|\epsilon_1|)$$ satisfies that $$\epsilon'< \frac{C_4}{R^2}$$ for a constant $C_4>0$ independent of $R\gg 1$.

First, let's consider the case when $Y(M,[g])> 0$. For any smooth $u:M\rightarrow \Bbb R$ such that $$Y(M,[g])\leq Y_{R^2g}(u)< Y(M,[g])+\frac{\epsilon}{4},$$
\begin{eqnarray*}
Y(M,[g_{_R}])&\leq& Y_{g_{_R}}(u)\\ &\leq& \frac{\int_M (6(1+\epsilon_3)|du|^2_{_{R^2g}}+(s_{_{R^2g}}+\epsilon_2)u^2)(1+\epsilon_1)\ d\mu_{_{R^2g}}}{(\int_M u^4(1+\epsilon_1)\ d\mu_{_{R^2g}})^{\frac{1}{2}}}\\ &=& \frac{\int_M (6|du|^2_{_{R^2g}}+s_{_{R^2g}}u^2)\ d\mu_{_{R^2g}}}{(\int_M u^4(1+\epsilon_1)\ d\mu_{_{R^2g}})^{\frac{1}{2}}}\\ & & +\frac{\int_{M} (6(\epsilon_1+\epsilon_3(1+\epsilon_1))|du|^2_{_{R^2g}}+(s_{_{R^2g}}\epsilon_1+\epsilon_2(1+\epsilon_1))u^2)\ d\mu_{_{R^2g}}}{(\int_M u^4(1+\epsilon_1)\ d\mu_{_{R^2g}})^{\frac{1}{2}}}\\ &\leq& \frac{Y_{R^2g}(u)}{(1-\epsilon')^{\frac{1}{2}}} +\frac{\epsilon'}{(1-\epsilon')^{\frac{1}{2}}}\frac{\int_{M_R} (6|du|^2_{_{R^2g}}+u^2)\ d\mu_{_{R^2g}}}{(\int_{M} u^4\ d\mu_{_{R^2g}})^{\frac{1}{2}}}\\ &\leq&  \frac{Y_{R^2g}(u)}{(1-\epsilon')^{\frac{1}{2}}}+\frac{\epsilon'}{(1-\epsilon')^{\frac{1}{2}}}\frac{\int_{M} (6|du|^2_{_{R^2g}}+s_{_{R^2g}}u^2+|s_{_{R^2g}}|u^2)\ d\mu_{_{R^2g}}+\int_{M_R}u^2d\mu_{_{R^2g}}}{(\int_{M} u^4\ d\mu_{_{R^2g}})^{\frac{1}{2}}}\\ &\leq&
\frac{Y_{R^2g}(u)}{(1-\epsilon')^{\frac{1}{2}}}+\frac{\epsilon'}{(1-\epsilon')^{\frac{1}{2}}}
\left(Y_{R^2g}(u)+(\int_{M}|s_{_{R^2g}}|^2d\mu_{_{R^2g}})^{\frac{1}{2}}+(\int_{M_R}d\mu_{_{R^2g}})^{\frac{1}{2}}\right)\\ &\leq&
\frac{Y(M,[g])+\frac{\epsilon}{4}}{(1-\epsilon')^{\frac{1}{2}}}+\frac{\epsilon'}{(1-\epsilon')^{\frac{1}{2}}}
\left(Y(M,[g])+\frac{\epsilon}{4}+C_5\right)\\
&<& Y(M,[g])+\frac{\epsilon}{2}
\end{eqnarray*}
for any sufficiently large $R$, where $C_5>0$ is a constant independent of $R>1$.

Secondly let's consider the case of $Y(M,[g])\leq 0$. For a Yamabe minimizer $u:M\rightarrow \Bbb R$ of $Y_g$ the same method as above gives
\begin{eqnarray*}
Y(M,[g_{_R}])&\leq& Y_{g_{_R}}(u)\\ &\leq&
\frac{Y_{R^2g}(u)}{(1+\epsilon')^{\frac{1}{2}}}+\frac{\epsilon'}{(1-\epsilon')^{\frac{1}{2}}}
\left(Y_{R^2g}(u)+(\int_{M}|s_{_{R^2g}}|^2d\mu_{_{R^2g}})^{\frac{1}{2}}+(\int_{M_R}d\mu_{_{R^2g}})^{\frac{1}{2}}\right)\\ &\leq& \frac{Y(M,[g])}{(1+\epsilon')^{\frac{1}{2}}}+\frac{\epsilon'}{(1-\epsilon')^{\frac{1}{2}}}
\left(Y(M,[g])+C_5\right)\\ &<& Y(M,[g])+\frac{\epsilon}{2}.
\end{eqnarray*}
for any sufficiently large $R$, where $Y_{R^2g}(u)$ is divided by $(1+\epsilon')^{\frac{1}{2}}$ rather than $(1-\epsilon')^{\frac{1}{2}}$, because $Y_{R^2g}(u)\leq 0$ now.

The proof of $$ Y(M,[g])<Y(M,[g_{_R}])+\frac{\epsilon}{2}$$ for any sufficiently large $R$ can be done in the same way as above just by changing the role of $R^2g$ and $g_{_R}$. We can write at each $x\in M$ $$d\mu_{_{R^2g}}(x)=(1+\varepsilon_1(x))d\mu_{g_{_R}}(x),\ \ \ \ \ s_{_{R^2g}}(x)=s_{g_{_R}}(x)+\varepsilon_2(x)$$
$$|\alpha|^2_{_{R^2g}}\leq (1+\varepsilon_3(x))|\alpha|^2_{g_{_R}}$$ for any $\alpha\in T_x^*M$ such that all $\varepsilon_i(x)$ are non-vanishing only for $x\in M_R$ and  $$\varepsilon':=\max_{x\in M} (|\varepsilon_1+(1+\varepsilon_1)\varepsilon_3|+|s_{g_{_R}}\varepsilon_1+(1+\varepsilon_1)\varepsilon_2|+|\varepsilon_1|)$$ satisfies that $$\varepsilon'< \frac{C_6}{R^2}$$ for a constant $C_6>0$ independent of $R\gg 1$.

Dealing with the positive and nonpositive case of $Y(M,[g_{_R}])$ together,
for any smooth $u_{_R}:M\rightarrow \Bbb R$ satisfying $$Y(M,[g_{_R}])\leq Y_{g_{_R}}(u_{_R})< Y(M,[g_{_R}])+\frac{\epsilon}{4}$$ we get
\begin{eqnarray*}
Y(M,[g])&\leq& Y_{R^2g}(u_{_R})\\ &\leq& \frac{\int_M (6(1+\varepsilon_3)|du_{_R}|^2_{g_{_R}}+(s_{g_{_R}}+\varepsilon_2)u_{_R}^2)(1+\varepsilon_1)\ d\mu_{g_{_R}}}{(\int_M u_{_R}^4(1+\varepsilon_1)\ d\mu_{g_{_R}})^{\frac{1}{2}}}\\ &=& \frac{\int_M (6|du_{_R}|^2_{g_{_R}}+s_{g_{_R}}u_{_R}^2)\ d\mu_{g_{_R}}}{(\int_M u_{_R}^4(1+\varepsilon_1)\ d\mu_{g_{_R}})^{\frac{1}{2}}}\\ & & +\frac{\int_{M} (6(\varepsilon_1+\varepsilon_3(1+\varepsilon_1))|du_{_R}|^2_{g_{_R}}+(s_{g_{_R}}\varepsilon_1+\varepsilon_2(1+\varepsilon_1))u_{_R}^2)\ d\mu_{g_{_R}}}{(\int_M u_{_R}^4(1+\varepsilon_1)\ d\mu_{g_{_R}})^{\frac{1}{2}}}\\ &\leq& \max(\frac{Y_{g_{_R}}(u_{_R})}{(1-\varepsilon')^{\frac{1}{2}}}, \frac{Y_{g_{_R}}(u_{_R})}{(1+\varepsilon')^{\frac{1}{2}}}) +\frac{\varepsilon'\int_{M_R} (6|du_{_R}|^2_{g_{_R}}+u_{_R}^2)\ d\mu_{g_{_R}}}{(1-\varepsilon')^{\frac{1}{2}}(\int_{M} u_{_R}^4\ d\mu_{g_{_R}})^{\frac{1}{2}}}\\ &\leq&  \max( \frac{Y_{g_{_R}}(u_{_R})}{(1-\varepsilon')^{\frac{1}{2}}}, \frac{Y_{g_{_R}}(u_{_R})}{(1+\varepsilon')^{\frac{1}{2}}}) \\ & & +  \frac{\varepsilon'}{(1-\varepsilon')^{\frac{1}{2}}}\frac{\int_{M} (6|du_{_R}|^2_{g_{_R}}+s_{g_{_R}}u_{_R}^2+|s_{g_{_R}}|u_{_R}^2)\ d\mu_{g_{_R}}+\int_{M_R}u_{_R}^2d\mu_{g_{_R}}}{(\int_{M} u_{_R}^4\ d\mu_{g_{_R}})^{\frac{1}{2}}}\\ &\leq&  \max( \frac{Y_{g_{_R}}(u_{_R})}{(1-\varepsilon')^{\frac{1}{2}}}, \frac{Y_{g_{_R}}(u_{_R})}{(1+\varepsilon')^{\frac{1}{2}}})\\ & & +\frac{\varepsilon'}{(1-\varepsilon')^{\frac{1}{2}}}\left(Y_{g_{_R}}(u_{_R})+(\int_{M}|s_{g_{_R}}|^2d\mu_{g_{_R}})^{\frac{1}{2}}      +(\int_{M_R}d\mu_{g_{_R}})^{\frac{1}{2}}\right)\\ &\leq&  \max( \frac{Y(M,[g_{_R}])+\frac{\epsilon}{4}}{(1-\varepsilon')^{\frac{1}{2}}}, \frac{Y(M,[g_{_R}])+\frac{\epsilon}{4}}{(1+\varepsilon')^{\frac{1}{2}}})+\frac{\varepsilon'(Y(M,[g_{_R}])+\frac{\varepsilon}{4}+C_7)}{(1-\varepsilon')^{\frac{1}{2}}}
\end{eqnarray*}
by using
\begin{eqnarray*}
\int_{M}|s_{g_{_R}}|^2d\mu_{g_{_R}}&=& \int_{M-M_R}|s_{_{R^2g}}|^2d\mu_{_{R^2g}}+ \int_{M_R}|s_{g_{_R}}|^2 d\mu_{g_{_R}}\\ &\leq& \int_{M-M_R}|s_g|^2d\mu_g+ \int_{M_R}C_8\ d\mu_{g_{_R}}\\  &\leq& C_9
\end{eqnarray*}
where $C_7, C_8, C_9$ are positive constants independent of $R>1$.

So we have that $Y(M,[g_{_R}])$ is bounded below by a (negative) constant independent of $R>1$.
Applying this fact back to the above last line of computing an upper bound of $Y(M,[g])$, it finally becomes $$<Y(M,[g_{_R}])+\frac{\epsilon}{2}$$ for any sufficiently large $R$.
This completes the proof.
\end{proof}

\subsection{Main proof}


Let $p_1,\cdots,p_m$ be all the orbifold points of $M$ with corresponding orbifold groups $\Gamma_1,\cdots,\Gamma_m$ respectively.
Observe that each $\Gamma_j$ has not only the induced isometric action on $(\Bbb R^4,g_{_{\Bbb E^4}})$ fixing the origin, but also it can act on $(S^4,g_{_{\Bbb S^4}})$ isometrically fixing only the south pole and the north pole $q_0$.
Let $\psi: \Bbb R^4\rightarrow \Bbb R$ be the conformal factor given by $\psi(x):=\frac{2}{1+|x|^2}$ so that $\psi^2g_{_{\Bbb E^4}}$ is equal to $g_{_{\Bbb S^4}}$ on $S^4-\{q_0\}$.

Let $\epsilon\in (0,Y(M,[g]))$. For each $j=1,\cdots,m$, take a smooth $\Gamma_j$-invariant metric $h_j$ on $\Bbb R^4$ such that $h_j$ coincides with $g_{_{\Bbb E^4}}$ outside of $B(1)$,
\begin{eqnarray}\label{Swiss}
||g_{_{\Bbb E^4}}-h_j||_{C^2}< \epsilon,\  \ \ \ \ \ \ \ s_{\psi^2h_j}>0,
\end{eqnarray}
and
$$0<\int_{\Bbb R^4} |W^+_{h_j}|^2d\mu_{h_j}<\frac{\epsilon}{2m}$$ where the $C^2$-norm is computed w.r.t. the Euclidean metric. (For example, $h_j$ can be obtained by taking any local small perturbation of $g_{_{\Bbb E^4}}$ along a free orbit such that $W^+_{h_j}$ is not identically zero.) We regard $\psi^2h_j$ as a smooth metric on $S^4$, which is $\Gamma_j$-invariant.
Since  $\psi^2{h_j}$ is not conformally flat and hence not conformal to $g_{_{\Bbb S^4}}$,
$$Y(S^4,[\psi^2h_j]_{\Gamma_j})< Y(S^4)$$ by the resolution of the Hevey-Vaugon conjecture in dimension 4.

For the orbifold metric $\hat{h}_j$ on $\Bbb R^4/\Gamma_j$ induced by $h_j$ and $\hat{\psi}: \Bbb R^4/\Gamma_j\rightarrow \Bbb R$ induced by $\Gamma_j$-invariant $\psi: \Bbb R^4\rightarrow \Bbb R$,
\begin{eqnarray}\label{motherlove}
Y(S^4/\Gamma_j,[\hat{\psi}^2\hat{h}_j])<\frac{Y(S^4)}{\sqrt{|\Gamma_j|}}\ \ \  \textrm{and}\ \ \ \int_{\Bbb R^4/\Gamma_j} |W^+_{\hat{h}_j}|^2d\mu_{\hat{h}_j}<\frac{\epsilon}{2m|\Gamma_j|}.
\end{eqnarray}
We take a smooth compact-supported function $\varphi_j:S^4/\Gamma_j-\{\hat{q}_0\}\rightarrow \Bbb R$  such that
$$Y_{\psi^2\hat{h}_j}(\varphi_j)<\frac{Y(S^4)}{\sqrt{|\Gamma_j|}}$$
where $\hat{q}_0$ is the orbifold point corresponding to $q_0\in S^4$, and take a conformal change
of $(S^4/\Gamma_j-\{\hat{q}_0\},\hat{\psi}^2\hat{h}_j)$  such that the end is isometric to a cylindrical end $$(S^3/\Gamma_j\times [0,\infty), g_{_{\Bbb S^3}}+dt^2)$$ while the metric remains the same as $\hat{\psi}^2\hat{h}_j$ on $\textrm{supp}(\varphi_j)$.
We chop off the part $S^3/\Gamma_j\times [\frac{l}{2},\infty)$ and take the remaining part with the resulting metric denoted by $\mathfrak{h}_j$.
We still have
\begin{eqnarray}\label{KCDC}
Y_{\mathfrak{h}_j}(\varphi_j)<\frac{Y(S^4)}{\sqrt{|\Gamma_j|}}.
\end{eqnarray}

On the $M$ side, we take a open neighborhood $V_j$ of each $p_j$ such that $V_1,\cdots,V_m$ are all mutually disjoint, and apply Lemma \ref{IKim} successively with $p_j, V_j$ and $\frac{\epsilon}{m}$ for $j=1,\cdots,m$. Let the resulting metric be $\bar{g}_\epsilon$ and it satisfies
\begin{eqnarray}\label{teacherlove}
|Y(M,[g])-Y(M,[\bar{g}_\epsilon])|<\frac{\epsilon}{2}\ \ \ \ \textrm{and}\ \ \ \ |\mathcal{W}_+(M,[g])-\mathcal{W}_+(M,[\bar{g}_\epsilon])|<\frac{\epsilon}{2}.
\end{eqnarray}
Using the conformally-flatness around each $p_j$, one can take a conformal change of $(M-\{p_1,\cdots,p_j\},\bar{g}_\epsilon)$ around each $p_j$ so that the end is isometric to a cylindrical end $(S^3/\Gamma_j\times[0,\infty), g_{_{\Bbb S^3}}+dt^2)$ and the resulting metric remains the same as $\bar{g}_\epsilon$ on $(\cup_{j=1}^mV_j)^c$.
We again chop off the part $S^3/\Gamma_j\times [\frac{l}{2},\infty)$ of each end and take the remaining bulk part $M'$.

Now we can glue all the spherical parts with the metric $\mathfrak{h}_j$ for $j=1,\cdots,m$ to $M'$ to get a smooth orbifold metric $g_\epsilon$ on $M$ having $m$ cylinders of length $l$.

Since each $\varphi_j$ can be viewed as a function on $(M,g_\epsilon)$,  by (\ref{KCDC}) $$Y_{g_\epsilon}(\varphi_j)<\frac{Y(S^4)}{\sqrt{|\Gamma_j|}},$$ which implies that $$Y(M, [g_\epsilon])\leq \min_j Y_{g_\epsilon}(\varphi_j)<\min_j \frac{Y(S^4)}{\sqrt{|\Gamma_j|}}.$$

The point of constructing $g_\epsilon$ instead of using $g$ is that the Yamabe problem is solvable for $g_\epsilon$ and $\mathcal{W}_+(M,[g_\epsilon])$ is close enough to $\mathcal{W}_+(M,[g])$ such that
\begin{eqnarray}\label{duck}
|\int_M |W^+_{g_{\epsilon}}|^2d\mu_{g_{\epsilon}}-\int_M |W^+_{g}|^2d\mu_{g}| &<& \epsilon
\end{eqnarray}
as a consequence of the 2nd inequalities of (\ref{motherlove}) and (\ref{teacherlove}).

Moreover $Y(M,[g_\epsilon])$ is positive when the cylinder length $l$ is sufficiently large, because
both of $Y(S^4/\Gamma_j,[\hat{\psi}^2\hat{h}_j])$ and $Y(M,[\bar{g}_\epsilon])$ are positive.
Our construction is a standard procedure of connect-summing two manifolds of positive Yamabe constant to produce a manifold of positive Yamabe constant.
This can be proved in the same way as is done when proving the connected sum theorem of Yamabe constant stating that
for any $\epsilon>0$ there exists a metric $g_{_L}$ on $M_1\#M_2$ such that
$$Y(M_1\#M_2, [g_{_L}])> Y(M_1\amalg M_2, [g_1\amalg g_2])-\epsilon.$$
Here $g_1\amalg g_2$ denotes the metric of the disjoint union of $(M_1,g_1)$ and $(M_2,g_2)$, and $g_{_L}$ is constructed out of $g_1$ and $g_2$ in the same way as we did to contain the cylinder of length $L$, and any $g_l$ for $l\geq L$ satisfies the property. For details, one may refer to \cite{koba2}.\footnote{The theorem is originally stated for manifolds, but the proof works well for orbifolds too. If at least one of $Y(M_1,[g_1])$ and $Y(M_2,[g_2])$ is positive, then $Y(M_1\amalg M_2, [g_1\amalg g_2])=\min(Y(M_1,[g_1]),Y(M_2,[g_2]))$.}

We always assume that such large $l$ is taken, and with the resulting $g_{\epsilon}$ we proceed to estimate $\mathcal{W}_+(M, [g])$.

\medskip

\textbf{\underline{Proof of (i,ii,iii,iv)}}

\begin{Lemma}
Let $\pi:\tilde{M}\rightarrow M$ be a $k$-fold orbifold covering. Then
$$2\pi^2 (2\chi_{orb}(M)+3\tau_{orb}(M))-
\int_M |W^+_{g}|^2d\mu_{g}\leq \frac{8\pi^2}{k|\tilde{\Gamma}|}$$ where $\tilde{\Gamma}$ is an orbifold group of $\tilde{M}$ with the largest order.
\end{Lemma}
\begin{proof}
By (\ref{duck}), it is enough to show that the above inequality holds for $(M,g_{\epsilon})$ for any sufficiently small $\epsilon$.
Let $x\in \tilde{M}$ be an orbifold  point corresponding to $\tilde{\Gamma}$, and $\pi(x)$ be $p_j$ for some $j$.(If $\tilde{M}$ is a manifold, $\tilde{\Gamma}$ is trivial.)
An important thing in the construction of $g_{\epsilon}$ is that one can take the neighborhood $V_j$ around $p_j$ arbitrarily small and we take $V_j$ contained in an evenly covered neighborhood of $p_j$ so that we can take a local uniformizing chart $U/\Gamma_j$ of $V_j$ where the covering projection $\pi$ is given by the obvious projection map $U/\tilde{\Gamma}\rightarrow U/\Gamma_j$.

For
\begin{eqnarray*}
d&:=&2\pi^2 (2\chi_{orb}(M)+3\tau_{orb}(M))-
\int_M |W^+_{g_{\epsilon}}|^2d\mu_{g_{\epsilon}},
\end{eqnarray*}
$$d=\int_M
(\frac{s_{g_{\epsilon}}^2}{48}-\frac{|\stackrel{\circ}r_{g_{\epsilon}}|^2}{4})\
d\mu_{g_{\epsilon}},$$ and hence
$$ \int_{\tilde{M}}
(\frac{s_{\pi^*{g_{\epsilon}}}^2}{48}-\frac{|\stackrel{\circ}r_{\pi^*{g_{\epsilon}}}|^2}{4})\
d\mu_{\pi^*{g_{\epsilon}}}=kd.$$ For any metric $\tilde{h}\in [\pi^*g_{\epsilon}]$ and the pullback orientation (or its reversed one) on $\tilde{M}$,
\begin{eqnarray*}
\int_{\tilde{M}}
(\frac{s_{\tilde{h}}^2}{48}-\frac{|\stackrel{\circ}r_{\tilde{h}}|^2}{4})\
d\mu_{\tilde{h}}&=& 2\pi^2 (2\chi_{orb}(\tilde{M})+3\tau_{orb}(\tilde{M}))-
\int_{\tilde{M}} |W^+_{\tilde{h}}|^2d\mu_{\tilde{h}}\\  &=& 2\pi^2 (2\chi_{orb}(\tilde{M})+3\tau_{orb}(\tilde{M}))-
\int_{\tilde{M}} |W^+_{\pi^*{g_{\epsilon}}}|^2d\mu_{g_{\epsilon}}\\ &=& \int_{\tilde{M}}
(\frac{s_{\pi^*{g_{\epsilon}}}^2}{48}-\frac{|\stackrel{\circ}r_{\pi^*g_{\epsilon}}|^2}{4})\
d\mu_{\pi^*g_{\epsilon}}\\ &=&  kd.
\end{eqnarray*}
Thus if $d\geq 0$, then
 $$\inf_{\tilde{h}\in [\pi^*g_{\epsilon}]} (\int_{\tilde{M}}
s_{\tilde{h}}^2\ d\mu_{\tilde{h}})^{\frac{1}{2}}\geq 4\sqrt{3kd}.$$

Since $Y(M,[g_{\epsilon}])$ is positive, so is $Y(\tilde{M},[\pi^*g_{\epsilon}])$.
We claim that the Yamabe problem is solvable on $(\tilde{M},\pi^*g_{\epsilon})$ too.
Since $\pi^*\varphi_j$ is a smooth function on $\tilde{M}$ and $(\textrm{supp}(\varphi_j),\mathfrak{h}_j)$ is contained in $V_j$ which is again contained in an evenly covered neighborhood of $p_j$,
\begin{eqnarray*}
Y(\tilde{M},[\pi^*g_{\epsilon}])&\leq& Y_{\pi^*g_{\epsilon}}(\pi^*\varphi_j)\\ &=& \left(\frac{|\Gamma_j|}{|\tilde{\Gamma}|}\right)^{\frac{1}{2}}Y_{g_{\epsilon}}(\varphi_j)\\ &=&  \left(\frac{|\Gamma_j|}{|\tilde{\Gamma}|}\right)^{\frac{1}{2}}Y_{\mathfrak{h}_j}(\varphi_j)\\ &<&\frac{Y(S^4)}{|\tilde{\Gamma}|^{\frac{1}{2}}}
\end{eqnarray*}
by (\ref{KCDC}).
Thus the Yamabe problem is solvable on $(\tilde{M},\pi^*g_{\epsilon})$, which enables us to have
$$Y(\tilde{M},[\pi^*g_{\epsilon}])=\inf_{\tilde{h}\in [\pi^*g_{\epsilon}]} (\int_{\tilde{M}}
s_{\tilde{h}}^2\ d\mu_{\tilde{h}})^{\frac{1}{2}}.$$

Therefore we conclude that $$d\leq \frac{Y(\tilde{M},[\pi^*g_{\epsilon}])^2}{48k}\leq  \frac{8\pi^2}{k|\tilde{\Gamma}|}$$
by using $|Y(\tilde{M},[\pi^*g_{\epsilon}])|< \frac{Y(S^4)}{\sqrt{|\tilde{\Gamma}|}}=\frac{8\sqrt{6}\pi}{\sqrt{|\tilde{\Gamma}|}}$.
\end{proof}

Now the proofs of (i), (ii) (iii) are obtained by applying the above lemma with $\tilde{M}$ respectively being the universal orbifold covering space of $M$, the ordinary universal covering space, an ordinary covering space of degree $|H_1(M,\Bbb Z)|$. In case of (ii) and (iii), the corresponding covering $\tilde{M}$ is an ordinary covering so that $\tilde{M}$ and $M$ have the same orbifold singularity types. That's why we wrote $\Gamma_j$ in the 3rd terms of the inequalities.

The case of (iv) is an immediate corollary of (i) and (ii).
In fact, if $\pi_1(M)$ contains a subgroup of arbitrarily large finite index,
$M$ has ordinary covering spaces of arbitrarily large finite degree, so $\pi_1^{orb}(M)$ has corresponding a subgroup of arbitrarily large finite index, because an ordinary covering of $M$ is also an orbifold covering of $M$.

\medskip

\textbf{\underline{Proof of (v)}}

First when $(M,[g])$ satisfies the strict generalized Aubin's inequality, the Yamabe problem is solvable on it and hence for its Yamabe metric $\check{g}$
\begin{eqnarray*}
2\chi_{orb}(M)+3\tau_{orb}(M) &=& \frac{1}{4\pi^2}\int_M
(\frac{s_{\check{g}}^2}{24}+2|W^+_{\check{g}}|^2-\frac{|\stackrel{\circ}r_{\check{g}}|^2}{2})\
d\mu_{\check{g}}\\ &\leq& \frac{1}{4\pi^2}\int_M
(\frac{s_{\check{g}}^2}{24}+2|W^+_{\check{g}}|^2)\
d\mu_{\check{g}}\\ &=&
\frac{1}{4\pi^2}(\frac{Y(M,[g])^2}{24}+\int_M2|W^+_{g}|^2\
d\mu_{g})\\
&\leq& \frac{3}{4\pi^2}\int_M|W^+_{g}|^2\
d\mu_{g}
\end{eqnarray*}
where we used Theorem \ref{th4} at the last step.

Secondly, when $(M,[g])$ saturates the generalized Aubin's inequality,
the Yamabe problem is solvable on $(M,[g_{\epsilon}])$ and for its Yamabe metric $\breve{g}$ we get
\begin{eqnarray*}
2\chi_{orb}(M)+3\tau_{orb}(M) &=& \frac{1}{4\pi^2}\int_M
(\frac{s_{\breve{g}}^2}{24}+2|W^+_{\breve{g}}|^2-\frac{|\stackrel{\circ}r_{\breve{g}}|^2}{2})\
d\mu_{\breve{g}}\\ &\leq& \frac{1}{4\pi^2}\int_M
(\frac{s_{\breve{g}}^2}{24}+2|W^+_{\breve{g}}|^2)\
d\mu_{\breve{g}}\\ &=&
\frac{1}{4\pi^2}(\frac{Y(M,[g_{\epsilon}])^2}{24}+\int_M2|W^+_{\breve{g}}|^2\
d\mu_{\breve{g}})\\ &\leq&
\frac{1}{4\pi^2}(\frac{Y(M,[g])^2}{24}+\int_M2|W^+_{g_{\epsilon}}|^2\
d\mu_{g_{\epsilon}})\\ &<&
\frac{1}{4\pi^2}(\frac{Y(M,[g])^2}{24}+\int_M2|W^+_{g}|^2\
d\mu_{g}+2\epsilon)\\
&\leq& \frac{3}{4\pi^2}\int_M|W^+_{g}|^2\
d\mu_{g}+\frac{\epsilon}{2\pi^2}
\end{eqnarray*}
where we used Theorem \ref{th4} at the last step. By letting $\epsilon\rightarrow 0$, we get the desired inequality.

To check the equality case, suppose that the equality holds.
Then from the proofs of the above two cases, it must hold that $$\int_M|W^+_{g}|^2\ d\mu_{g}=\frac{1}{24}(Y(M,[g]))^2$$ and
by Theorem \ref{th4} $[g]$ has a K\"ahler Yamabe metric, say $\hat{g}$. Now
\begin{eqnarray*}
2\chi_{orb}(M)+3\tau_{orb}(M) &=& \frac{1}{4\pi^2}\int_M
(\frac{s_{\hat{g}}^2}{24}+2|W^+_{\hat{g}}|^2-\frac{|\stackrel{\circ}r_{\hat{g}}|^2}{2})\
d\mu_{\hat{g}}\\ &\leq& \frac{1}{4\pi^2}\int_M
(\frac{s_{\hat{g}}^2}{24}+2|W^+_{\hat{g}}|^2)\ d\mu_{\hat{g}}\\ &=& \frac{1}{4\pi^2}(\frac{1}{24}Y(M,[g])^2+\int_M
2|W^+_{g}|^2\ d\mu_{g})\\
&=& \frac{3}{4\pi^2}\int_M|W^+_{g}|^2\ d\mu_{g},
\end{eqnarray*}
implying that $\stackrel{\circ}r_{\hat{g}}=0$, i.e. $\hat{g}$ is Einstein.

Conversely, if $\hat{g}\in [g]$ is K\"ahler-Einstein, then
$\stackrel{\circ}r_{\hat{g}}=0$ and $s_{\hat{g}}=2\sqrt{6}|W^+_{\hat{g}}|$, from which
it follows that
\begin{eqnarray*}
2\chi_{orb}(M)+3\tau_{orb}(M) &=& \frac{1}{4\pi^2}\int_M
(\frac{s_{\hat{g}}^2}{24}+2|W^+_{\hat{g}}|^2-\frac{|\stackrel{\circ}r_{\hat{g}}|^2}{2})\
d\mu_{\hat{g}}\\ &=& \frac{3}{4\pi^2}\int_M|W^+_{\hat{g}}|^2\
d\mu_{\hat{g}}\\ &=& \frac{3}{4\pi^2}\int_M|W^+_{g}|^2\
d\mu_{g}.
\end{eqnarray*}
This completes the proof.

\medskip

\textbf{\underline{Proof of (vi)}}

The proof of the inequality is proved in the same way as the case (v).
First when $(M,[g])$ satisfies the strict generalized Aubin's inequality, the Yamabe problem is solvable on it and hence for its Yamabe metric $\check{g}$
\begin{eqnarray*}
2\chi_{orb}(M)+3\tau_{orb}(M) &=& \frac{1}{4\pi^2}\int_M
(\frac{s_{\check{g}}^2}{24}+2|W^+_{\check{g}}|^2-\frac{|\stackrel{\circ}r_{\check{g}}|^2}{2})\
d\mu_{\check{g}}\\ &\leq& \frac{1}{4\pi^2}\int_M
(\frac{s_{\check{g}}^2}{24}+2|W^+_{\check{g}}|^2)\
d\mu_{\check{g}}\\ &=&
\frac{1}{4\pi^2}(\frac{Y(M,[g])^2}{24}+\int_M2|W^+_{g}|^2\
d\mu_{g})\\
&\leq& \frac{3}{4\pi^2}\int_M|W^+_{g}|^2\
d\mu_{g}
\end{eqnarray*}
where we used Theorem \ref{th5} at the last step.

Secondly, when $(M,[g])$ saturates the generalized Aubin's inequality, the Yamabe problem is solvable on $(M,[g_{\epsilon}])$ and for its Yamabe metric $\breve{g}$ we get
\begin{eqnarray*}
2\chi_{orb}(M)+3\tau_{orb}(M) &=& \frac{1}{4\pi^2}\int_M
(\frac{s_{\breve{g}}^2}{24}+2|W^+_{\breve{g}}|^2-\frac{|\stackrel{\circ}r_{\breve{g}}|^2}{2})\
d\mu_{\breve{g}}\\ &\leq& \frac{1}{4\pi^2}\int_M
(\frac{s_{\breve{g}}^2}{24}+2|W^+_{\breve{g}}|^2)\
d\mu_{\breve{g}}\\ &=&
\frac{1}{4\pi^2}(\frac{Y(M,[g_{\epsilon}])^2}{24}+\int_M2|W^+_{\breve{g}}|^2\
d\mu_{\breve{g}})\\ &<&
\frac{1}{4\pi^2}(\frac{Y(M,[g])^2}{24}+\int_M2|W^+_{g}|^2\
d\mu_{g})+2\epsilon)\\
&\leq& \frac{3}{4\pi^2}\int_M|W^+_{g}|^2\
d\mu_{g}+\frac{\epsilon}{2\pi^2}
\end{eqnarray*}
where we used Theorem \ref{th5} at the last step. By letting $\epsilon\rightarrow 0$, we get the desired inequality.

To identify the equality cases, suppose the equality holds.
Then from the above proof, it must hold that $$\frac{(Y(M,[g]))^2}{24}=\int_M|W^+_{g}|^2\
d\mu_{g},$$ and then by Theorem \ref{th5} $g$ is a Yamabe metric with positive constant scalar curvature and $W_g^+$ is nonzero parallel with $$|\textrm{Spec}(W_g^+)|=2$$ at each point.
Since the Yamabe problem is now solvable on $(M,[g])$,
\begin{eqnarray*}
2\chi_{orb}(M)+3\tau_{orb}(M) &=& \frac{1}{4\pi^2}\int_M
(\frac{s_{g}^2}{24}+2|W^+_{g}|^2-\frac{|\stackrel{\circ}r_{g}|^2}{2})\
d\mu_{g}\\ &\leq& \frac{1}{4\pi^2}\int_M
(\frac{s_{g}^2}{24}+2|W^+_{g}|^2)\
d\mu_{\tilde{g}}\\ &=&
\frac{1}{4\pi^2}(\frac{Y(M,[g])^2}{24}+\int_M2|W^+_{g}|^2\
d\mu_{g})\\ &=&
\frac{3}{4\pi^2}\int_M|W^+_{g}|^2\
d\mu_{g},
\end{eqnarray*}
from which it follows that $\stackrel{\circ}r_{g}\equiv 0$.

Now we can apply Proposition 5 of \cite{derdzinski} which asserts that oriented Riemannian 4-manifold whose self-dual Weyl tensor $W_g^+$ is harmonic with at most 2 distinct eigenvalues at each point must be locally conformally K\"ahler.
Indeed $|W_g^+|^{\frac{2}{3}}g$ is locally K\"ahler with K\"ahler form $|W_g^+|^{\frac{2}{3}}\omega$ where $\omega$ is an eigenvector of $W_g^+$ with $|\omega|_g=\sqrt{2}$.
Since this is a local statement, it must hold for an orbifold $(M,g)$ too, and moreover the conformal factor is constant in our case.

Since there are locally only two choice of continuous $\omega$, the monodromy around any loop is in $\Bbb Z_2$. Therefore if the monodromy is trivial, then the K\"ahler form is globally well-defined on $(M,g)$, and otherwise it is well-defined on the double cover corresponding to the nontrivial monodromy. In the latter case, the covering transformation is isometric but anti-holomorphic.

Conversely, if $(M,g)$ or its double cover is a K\"ahler-Einstein orbifold, the equality is obtained by applying the case (v) to $(M,g)$ or the double cover of it.
This completes the proof.

\begin{Remark}
We could not characterize the cases where the inequality of (i) (or (ii)) is saturated. It is certain from the above proof that if its universal orbifold (or ordinary) cover admits an Einstein orbifold metric with Yamabe constant saturating the generalized Aubin's inequality, it certainly attains the equality.
However the classification of 4-orbifolds saturating the generalized Aubin's inequality is not obtained yet. There do exist such 4-orbifolds other than the quotients of a round 4-sphere.

Theorem \ref{th3} also holds when $Y(M,[g])=0$. But in that case we have a better lower bound given in (\ref{ggg-1}).

\end{Remark}

\section{Application and examples}

\subsection{Application to Einstein metric}

The following application to Einstein 4-orbifolds can be proved in the same way as the manifold case in \cite{gur}.
\begin{Corollary}\label{th3-1}
Let $M$ be a smooth closed oriented 4-orbifold and $g_t$ for $t\in (-a,a)$ be a smooth family of Einstein orbifold metrics with nonnegative scalar curvature on $M$. If $g_0$ is K\"ahler, then so is any other $g_t$.
\end{Corollary}
\begin{proof}
Define a smooth function $F:(-a,a)\rightarrow \Bbb R$ by $$F(t)=\frac{(\int_Ms_{g_t}d\mu_{g_t})^2}{\int_Md\mu_{g_t}}=\int_Ms^2_{g_t}d\mu_{g_t}.$$
Since $g_t$ is Einstein, it is a critical point of the normalized Einstein-Hilbert action $\mathfrak{F}$. (This is originally stated for a smooth closed manifold by Hilbert and its proof in \cite{Be} still works well for an orbifold.) So $F'(t)$ for all $t$ is zero, meaning that $F$ is constant.
Combining it with
\begin{eqnarray*}
2\pi^2 (2\chi_{orb}(M)+3\tau_{orb}(M))&=&\int_M (|W^+_{g_t}|^2+\frac{s_{g_t}^2}{48}-\frac{|\stackrel{\circ}r_{g_t}|^2}{4})\
d\mu_{g_t}\\ &=&\int_M (|W^+_{g_t}|^2+\frac{s_{g_t}^2}{48})\ d\mu_{g_t}\\ &=& \mathcal{W}_+(M,[g_t])+\frac{F(t)}{48},
\end{eqnarray*}
we get that $\mathcal{W}_+(M,[g_t])$ is constant too.

Since $g_0$ is K\"ahler-Einstein,  by (v) of Theorem \ref{th3}
\begin{eqnarray*}
\frac{4\pi^2}{3}(2\chi_{orb}(M)+3\tau_{orb}(M))&=& \mathcal{W}_+(M,[g_0])\\ &=& \mathcal{W}_+(M,[g_t]),
\end{eqnarray*}
and hence $g_t$ is conformal to a K\"ahler-Einstein orbifold metric. We need to show that this conformal factor is constant.

Since each $g_t$ is Einstein and the equality holds in (vi) of Theorem \ref{th3}, $(M,g_t)$ or its double cover with pull-back metric is K\"ahler.
It only remains to deal with the latter case. 
In that case the double cover now has two conformally equivalent metrics which are K\"ahler-Einstein. These two orbifold metrics must be homothetic by the following general fact :

Two K\"ahler metrics of constant scalar curvature on a 4-orbifold are conformally equivalent iff they are homothetic.\footnote{It's because any K\"ahler metric on a 4-orbifold satisfies $|s|=2\sqrt{6}|W^+|$.}
\end{proof}

\subsection{Self-dual orbifold with positive scalar curvature}

Theorem \ref{th3} gives topological constraints for the existence of a self-dual orbifold metric with positive Yamabe constant.

\begin{Theorem}\label{Gan-YH-1}
Let $(M,g)$ be a smooth closed oriented self-dual 4-orbifold having $Y(M,[g])>0$.
\begin{description}
 \item[(i)] If $b_2^+(M)>0$, then $$\chi_{orb}(M)\leq 3\tau_{orb}(M)$$ where the equality  holds iff $g$ is conformal to an orbifold K\"ahler-Einstein  metric.
 \item[(ii)] If $\pi_1(M)$ or $\pi_1^{orb}(M)$ contains a subgroup of arbitrarily large finite index, then $$2\chi_{orb}(M)\leq 3\tau_{orb}(M).$$
 \item[(iii)] If $\delta_gW_g^+=0$ for nonzero $W_g^+$, then $$\chi_{orb}(M)\leq 3\tau_{orb}(M)$$ where the equality holds iff $(M,g)$ is K\"ahler or the quotient of a K\"ahler orbifold by a free anti-holomorphic isometric involution.
\end{description}
\end{Theorem}
\begin{proof}
The proof is almost immediate from Theorem \ref{th3}. If $M$ admits a self-dual metric $g$, then from (\ref{knuefaculty}) we have $$\mathcal{W}_+(M,[g])= 12\pi^2\tau_{orb}(M).$$ To prove (i), we apply (v) of Theorem \ref{th3} and get
\begin{eqnarray}\label{son&kane}
12\pi^2\tau_{orb}(M)\geq \frac{4\pi^2}{3}(2\chi_{orb}(M)+3\tau_{orb}(M))
\end{eqnarray}
which simplifies to the desired inequality, and the equality condition is inherited from that of (v) of Theorem \ref{th3}.

(ii) and (iii) can be proved in the same way by using (iv) and (vi) of Theorem \ref{th3} respectively.

\end{proof}

\begin{Corollary}\label{cor1}
Let $M$ be a smooth closed oriented 4-orbifold. Suppose that either $b_2^+(M)>0$ or that $\pi_1(M)$ or $\pi_1^{orb}(M)$ contains a subgroup of arbitrarily large finite index. Then there exists $n_0\in \Bbb N$ such that $M_n:=M\# n(S^2\times S^2)$ for any $n\geq n_0$ never admits a self-dual orbifold metric of positive Yamabe constant.
\end{Corollary}
\begin{proof}
This is proved by the direct application of (i) and (ii) of the above theorem.
By using the computation of $\chi$ and $\tau$ under a connected sum,
\begin{eqnarray*}
\chi_{orb}(M_n)- 3\tau_{orb}(M_n) &=& (\chi_{orb}(M)+n\chi_{orb}(S^2\times S^2)-2n)\\ & & -3(\tau_{orb}(M)+\tau_{orb}(S^2\times S^2))\\
&=& \chi_{orb}(M)-3\tau_{orb}(M)+2n
\end{eqnarray*}
and
\begin{eqnarray*}
2\chi_{orb}(M_n)- 3\tau_{orb}(M_n)
&=& 2\chi_{orb}(M)-3\tau_{orb}(M)+4n,
\end{eqnarray*}
both of which are positive for any sufficiently large $n$.
\end{proof}

\begin{Example}\label{Pastor-HYJ}
Consider a $\Bbb Z_2$-action on $S^4=\{(x_1,\cdots,x_5)\in \Bbb R^5|\sum_ix_i^2=1\}$ given by the orientation-preserving diffeomorphism $$(x_1,\cdots,x_5)\mapsto (-x_1,\cdots,-x_4,x_5).$$
It's quotient  $S^4/\Bbb Z_2$ is a smooth oriented orbifold with 2 orbifold points.
Also consider  a $\Bbb Z_2$-action on $S^2\times S^2=\{(x_1,\cdots,x_6)\in \Bbb R^6|\sum_{i=1}^3x_i^2=\sum_{i=4}^6x_i^2=1\}$
given by the orientation-preserving diffeomorphism $$(x_1,\cdots,x_6)\mapsto (-x_1,-x_2,x_3,-x_4,-x_5,x_6).$$
It's quotient  $(S^2\times S^2)/\Bbb Z_2$ is also a smooth oriented orbifold with 4 orbifold points.

We claim that the connected sum $$X:=S^4/\Bbb Z_2\ \# \ (S^2\times S^2)/\Bbb Z_2$$ never admits a self-dual orbifold metric of positive scalar curvature, although  $S^4/\Bbb Z_2$ does obviously.
By the Seifert–Van Kampen theorem which still holds for the orbifold fundamental group \cite{SChoi},
\begin{eqnarray*}
\pi_1^{orb}(X)&=& \pi_1^{orb}(S^4/\Bbb Z_2)*\pi_1^{orb}((S^2\times S^2)/\Bbb Z_2)\\ &=& \Bbb Z_2*\Bbb Z_2
\end{eqnarray*}
which obviously contains a subgroup of arbitrarily large finite index.\footnote{In $\Bbb Z_2*\Bbb Z_2=\langle a,b|a^2=b^2=1\rangle$, the subgroup generated by $(ab)^d$ for $d\in \Bbb N$ has index $2d$.}

By using the formula (\ref{Gan-YH})
\begin{eqnarray*}
2\chi_{orb}(X)-3\tau_{orb}(X)&=& 2(\chi_{orb}(S^4/\Bbb Z_2)+\chi_{orb}((S^2\times S^2)/\Bbb Z_2)-2)\\ & &-3(\tau_{orb}(S^4/\Bbb Z_2)+\tau_{orb}((S^2\times S^2)/\Bbb Z_2))\\
&=& 2(1+2-2)-3(0+0)\\ &>& 0.
\end{eqnarray*}
Therefore $X$ cannot support a self-dual orbifold metric of positive Yamabe constant by (ii) of Theorem \ref{Gan-YH-1}, and hence the desired conclusion obviously follows.

Nevertheless $X$ has an orbifold metric of positive scalar curvature, because $S^4/\Bbb Z_2$ admits a (conformally-flat) orbifold metric of positive scalar curvature and $(S^2\times S^2)/\Bbb Z_2$ admits a (K\"ahler-Einstein) orbifold metric $g_{_{KE}}$ with positive scalar curvature.

Let's try to estimate $\nu_+(X)$. By Theorem \ref{th3} (iv)
\begin{eqnarray*}
\nu_+(X)&\geq& -12\pi^2\tau_{orb}(X)+2\cdot 2\pi^2(2\chi_{orb}(X)+3\tau_{orb}(X))\\ &=& 8\pi^2.
\end{eqnarray*}

Note that for any $n$-orbifolds $M_1$ and $M_2$
\begin{eqnarray}\label{vaccine-president}
\nu_+(M_1\# M_2)\leq \nu_+(M_1)+\nu_+(M_2)
\end{eqnarray}
which can be proved in the same way as the $\nu$ invariant \cite{koba}.

By using this formula and  Theorem \ref{th3} (v)
\begin{eqnarray*}
\nu_+(X)&\leq& \nu_+(S^4/\Bbb Z_2)+\nu_+((S^2\times S^2)/\Bbb Z_2) \\ &=& 0-12\pi^2\tau_{orb}((S^2\times S^2)/\Bbb Z_2)+
2\mathcal{W}_+((S^2\times S^2)/\Bbb Z_2,[g_{_{KE}}])\\ &=&\frac{32\pi^2}{3}.
\end{eqnarray*}
It remains a problem to exactly compute $\nu_+(M)$ and $\nu(M)$.
\end{Example}

\subsection{Optimality of our estimate of $\mathcal{W}_+$}

Here we show that there exist examples where our estimate of $\mathcal{W}_+$ gives a better optimized lower bound of $\mathcal{W}_+$ than any previously known bounds.
We will do this by constructing a 4-manifold $(M,g)$ such that $Y(M,[g])$ is positive and $\pi_1(M)$ contains a subgroup of arbitrarily large finite index so that $2\pi^2(2\chi(M)+3\tau(M))$ of Theorem \ref{th2} (iii) is applicable but it has $b_1(M)=0$.
So without our estimate the inequalities available to such $M$ are
(\ref{thank}), the part (i) of Gursky's Theorem \ref{th1}, and
\begin{eqnarray}\label{Hirze}
\mathcal{W}_+(M,[g])\geq\mathcal{W}_+(M,[g])-\mathcal{W}_-(M,[g])= 12\pi^2\tau(M)
\end{eqnarray}
which is the equivalent of (\ref{DSLee}).
And if $b_2^-(M)>0$, one can also apply the part (i) of Theorem \ref{th1} to the reverse-oriented manifold $\overline{M}$ to get
$$\mathcal{W}_-(M,[g])\geq \frac{4\pi^2}{3}(2\chi(M)-3\tau(M))$$
so that
\begin{eqnarray}\label{snufaculty}
\mathcal{W}_+(M,[g])&=& 12\pi^2\tau(M)+\mathcal{W}_-(M,[g])\nonumber\\ &\geq& \frac{8\pi^2}{3}(\chi(M)+3\tau(M)).
\end{eqnarray}
This gives a better bound than the bound (\ref{thank}) which says that
$$\mathcal{W}_+(M,[g])\geq 2\pi^2(\chi(M)+3\tau(M)).$$

Thus we need to check when the bound $2\pi^2(2\chi(M)+3\tau(M))$ is strictly greater than those in the part (i) of Theorem \ref{th1}, (\ref{Hirze}), and (\ref{snufaculty}).
Since
$$2\pi^2(2\chi(M)+3\tau(M)) >\frac{4\pi^2}{3}(2\chi(M)+3\tau(M))\ \ \ \Longleftrightarrow\ \ \  2\chi(M)+3\tau(M)> 0,$$
$$ 2\pi^2(2\chi(M)+3\tau(M))> 12\pi^2\tau(M)\ \ \ \Longleftrightarrow\ \ \ 2\chi(M)-3\tau(M)> 0,$$
$$ 2\pi^2(2\chi(M)+3\tau(M))> \frac{8\pi^2}{3}(\chi(M)+3\tau(M))\ \ \ \Longleftrightarrow\ \ \ 2\chi(M)-3\tau(M)> 0,$$
the condition
\begin{eqnarray}\label{GanYH-0}
2\chi(M)\pm 3\tau(M)>0
\end{eqnarray}
is all-sufficient.

We now construct an explicit example of smooth closed oriented Riemannian 4-manifold $(M,g)$ satisfying $Y(M,[g])>0$, $2\chi(M)\pm 3\tau(M)>0$, and $b_1(M)=0$,
while $\pi_1(M)$ contains a subgroup of arbitrarily large finite index.
Recall that $S^2\times S^2$ with the standard product metric has a smooth free orientation-preserving isometric $\Bbb Z_2$-action given by antipodal maps of each $S^2$ factor. Then the connected sum $$N':=(S^2\times S^2)\# 2N$$ for any oriented 4-manifold $N$ where two $N$ are glued around any $\Bbb Z_2$-orbit also has an induced $\Bbb Z_2$-action in the obvious way. We require that $N$ is simply-connected and has a metric of positive scalar curvature so that we can endow $N'$ with a metric of positive scalar curvature by the well-known Gromov-Lawson surgery \cite{Gromov}. By doing the surgery in the $\Bbb Z_2$-equivariant way, the $\Bbb Z_2$-action on $N'$ can be made isometric too.

Let $N''$ be the quotient of $N'$ by this $\Bbb Z_2$-action and define $M$ to be the connected sum $kN''$ for any $k\geq 2$. Since $N''$ admits a metric of positive scalar curvature, so does $M$ again by the Gromov-Lawson theorem. By the Seifert–Van Kampen theorem $\pi_1(M)$ is the $k$-fold free product $\Bbb Z_2*\cdots *\Bbb Z_2$ of $\Bbb Z_2$, and hence $b_1(M)$ is 0. Because $k\geq 2$, $\pi_1(M)$  has a subgroup of arbitrarily large finite index.\footnote{When $k=2$, we have already seen it in the previous subsection. Similarly when $k\geq 3$, if $a_i$ denotes a generator of the $i$-th $\Bbb Z_2$, then the subgroup generated by $\{(a_1a_2)^d,a_3,\cdots,a_k\}$ for $d\in \Bbb N$ has index $2d$.}
A simple computation shows
\begin{eqnarray*}
2\chi(M)\pm 3\tau(M)&=& 2(k\chi(N'')-(k-1)2)\pm 3k\tau(N'')\\ &=& k(2\chi(N'')\pm 3\tau(N''))-4(k-1)\\&=&
\frac{k}{2}(2\chi(N')\pm 3\tau(N'))-4(k-1)\\ &=& \frac{k}{2}(2(4+2\chi(N)-2\cdot 2)\pm 3(0+2\tau(N)))-4(k-1)\\ &=& k(2\chi(N)\pm 3\tau(N))-4(k-1).
\end{eqnarray*}

For $M$ to satisfy (\ref{GanYH-0}), one may take $N$ to be $$n(S^2\times S^2)\ \ \ \ \textrm{or}\ \ \ \ l\Bbb CP^2\# m\overline{\Bbb CP}^2$$ for any nonnegative integers $n,l,m$ satisfying $\frac{l}{5}-\frac{4}{5k}< m < 5l+\frac{4}{k}$, both of which have positive scalar curvature metrics.
If $l+m\geq 1$, $S^2\times S^2\# l\Bbb CP^2\# m\overline{\Bbb CP}^2$ is diffeomorphic to $(l+1)\Bbb CP^2\# (m+1)\overline{\Bbb CP}^2$, and hence $M$ is diffeormophic to the $k$-fold connected sum of $((2l+1)\Bbb CP^2\# (2m+1)\overline{\Bbb CP}^2)/\Bbb Z_2$.

\begin{Remark}
Since the above manifolds $M$ satisfy (\ref{GanYH-0}), $M$ never supports a self-dual metric of positive Yamabe constant by (ii) of Theorem \ref{Gan-YH-1}.
\end{Remark}

\subsection{Some exact computations of $\nu$ and $\nu_+$}\label{coffee2u}

A basic inequality for $\nu$ and $\nu_+$ following from (\ref{knuefaculty}) is
\begin{eqnarray}\label{DSLee-999}
\nu_+(M)\geq \nu(M)\geq 12\pi^2\tau_{orb}(M),
\end{eqnarray}
where $\nu(M)= 12\pi^2\tau_{orb}(M)$ holds if $M$ admits a self-dual orbifold metric and
$\nu_+(M)= 12\pi^2\tau_{orb}(M)$  if $M$ admits a self-dual orbifold metric of positive Yamabe constant.

\begin{Theorem}
Let $M_1,\cdots,M_k$ and $X$ be smooth closed oriented 4-orbifolds such that all $M_i$ admit self-dual metrics and $\nu(X)=\tau_{orb}(X)=0$.
Then $M:=M_1\# \cdots\# M_k\# X$ has $$\nu(M)=12\pi^2\sum_i\tau_{orb}(M_i).$$
Furthermore if all $M_i$ admit self-dual metrics of positive Yamabe constant and $\nu_+(X)$ is 0, then $\nu_+(M)$ has the same value.
\end{Theorem}
\begin{proof}
By applying (\ref{DSLee-999}) to self-dual oribifold $M_i$,
$$\nu(M_i)=12\pi^2\tau_{orb}(M_i),$$ and
\begin{eqnarray*}
\nu(M)&\geq& 12\pi^2\tau_{orb}(M) \\&=&  12\pi^2(\sum_i\tau_{orb}(M_i)+\tau_{orb}(X))\\ &=& 12\pi^2\sum_i\tau_{orb}(M_i).
\end{eqnarray*}
On the other hand by the connected sum formula (\ref{vaccine-president})
\begin{eqnarray*}
\nu(M)&\leq& \sum_i\nu(M_i)+\nu(X)\\ &=& 12\pi^2\sum_i\tau_{orb}(M_i).
\end{eqnarray*}
The 2nd statement is proved in the same way.
\end{proof}

For such example of $X$ with $\nu_+(X)=\tau_{orb}(X)=0$, there are
$$S^1\times S^3/G\ \ \ \textrm{or}\ \ \  S^4/G'\ \ \  \textrm{or}\ \ \ S^2\times T^2$$ for 3-dimensional spherical space form groups $G$ and $G'$, because the first two have conformally-flat metrics with positive scalar curvature and the third one collapses with curvature bounded and scalar curvature positive. In fact any connected sum of those has the same property. Thus by using the above theorem and the fact that the Fubini-Study metric of $\Bbb CP^2$ is a self-dual metric of positive scalar curvature,
$$\nu(k\Bbb CP^2\#l(S^4/G')\#m(S^1\times  S^3/G)\# n(S^2\times T^2))=12k\pi^2$$ and $\nu_+$ has the same value.

To find more examples, we briefly review some self-dual oribfolds. First we consider the ALE gravitational instanton $X_\Gamma$.
For $\Gamma\subset SU(2)$ corresponding to the A-D-E root system, $X_\Gamma$ is defined as the minimal resolution of an orbifold $\Bbb C^2/\Gamma$, and it admits a complete Riemannian metric $g_{_\Gamma}$ with holonomy $SU(2)$. There exists its conformal compactification $\hat{X}_\Gamma$ as a smooth orbifold.

Since any scalar-flat K\"ahler metric is anti-self-dual, $X_\Gamma$ and $\hat{X}_\Gamma$ with the orientation opposite to the complex one is self-dual, and the explicit computation of $\mathcal{W}_+$ of them can be easily read from
\begin{eqnarray*}
\nu(\hat{X}_\Gamma) &=& \mathcal{W}_+(X_\Gamma,[g_{_\Gamma}])\\ &=&\int_{X_\Gamma}|R|^2d\mu_{g_{_\Gamma}}\\ &=& 8\pi^2(\chi(X_\Gamma)-\frac{1}{|\Gamma|})\\ &=& 8\pi^2(|\Gamma|-\frac{1}{|\Gamma|})
\end{eqnarray*}
by using the vanishing of all other parts of Riemann curvature tensor $R$ of $g_{_\Gamma}$ and the Chern-Gauss-Bonnet formula on an ALE manifold.
By the result of Viaclovsky \cite{viaclo2}  $Y(\hat{X}_\Gamma, [\hat{g}_{_\Gamma}])$ for an orbifold metric $\hat{g}_{_\Gamma}$ conformal to $g_{_\Gamma}$ is exactly equal to $Y(S^4/\Gamma)=\frac{Y(S^4)}{\sqrt{|\Gamma|}}$ saturating the generalized Aubin's inequality.

For the 2nd example, let's consider the total space $L_n$ of complex line bundle over $\Bbb CP^1$ with the 1st Chern number $-n\leq -1$. Due to LeBrun \cite{Le}, $L_n$ carries an explicit ALE K\"ahler metric $g_n$ with zero scalar curvature. In this case too, there exists its conformal compactification $\hat{L}_n$ as a smooth orbifold. The orbifold group of the unique orbifold point is $\Bbb Z_n$ generated by the map given by $$(z_1,z_2)\mapsto (e^{\frac{2\pi i}{n}}z_1,e^{-\frac{2\pi i}{n}}z_2)$$ (See \cite{viaclo1}.)

To compute $\eta(S^3/\Bbb Z_n)$ for such a $\Bbb Z_n$-action, one can use the formula from \cite{APS, Lock2, naka}
\begin{eqnarray*}
\eta(S^3/\Bbb Z_n)&=& -\frac{1}{n}\sum_{k=1}^{n-1}\cot^2(\frac{k\pi}{n})\\ &=& -\frac{(n-1)(n-2)}{3n}.
\end{eqnarray*}
Since $g_n$ is scalar-flat, $Y(\hat{L}_n,[\hat{g}])$ for an orbifold metric $\hat{g}$ conformal to $g_n$ on $L_n$ is positive.
With respect to the reverse orientation, $\hat{g}$ is self-dual and hence
\begin{eqnarray*}
\nu(\hat{L}_n) &=& 12\pi^2\tau_{orb}(\hat{L}_n)\\ &=& 12\pi^2(1+\frac{(n-1)(n-2)}{3n}).
\end{eqnarray*}
(Recall that eta invariant changes sign when reversing the orientation.)
In fact this self-dual orbifold turns out to be the weighted projective plane $\Bbb CP^2_{(1,1,n)}$.

The weighted projective plane $\Bbb CP^2_{(d_1,d_2,d_3)}$ for coprime positive integers $d_i$ is defined as the quotient of $\Bbb C^3-\{0\}$ by the $\Bbb C^*$-action $$\lambda\cdot(z_1,z_2,z_3)=(\lambda^{d_1}z_1,\lambda^{d_2}z_2,\lambda^{d_3}z_3), \ \ \ \ \forall \lambda\in \Bbb C^*:=\Bbb C-\{0\},$$ so it has 3 orbifold points $[1:0:0], [0:1:0], [0:0:1]$ with orbifold groups $\Bbb Z_{d_1}, \Bbb Z_{d_2}, \Bbb Z_{d_3}$ respectively.
It is simply-connected (in fact $\pi_1^{orb}=0$), and its cohomology is isomorphic to that of $\Bbb CP^2$ so that $\chi=3$ and $\tau=1$.
By the well-known computations \cite{viaclo3}, its eta invariant term is given by
$$\sum_i\eta(S^3/\Bbb Z_{d_i})= -1+\frac{d_1^2+d_2^2+d_3^2}{3d_1d_2d_3}.$$
By the result of Bryant \cite{Bry} any weighted projective plane supports a self-dual K\"ahler metric,\footnote{Many of them are found to have positive Yamabe constant. See \cite{Bry, viaclo3}.}  so
\begin{eqnarray*}
\nu(\Bbb CP^2_{(d_1,d_2,d_3)})&=& 12\pi^2\tau_{orb}(\Bbb CP^2_{(d_1,d_2,d_3)})\\ &=& 4\pi^2\frac{d_1^2+d_2^2+d_3^2}{d_1d_2d_3}.
\end{eqnarray*}

In summary, one can compute the $\nu$ and $\nu_+$ invariants of any connected sum of these self-dual orbifolds and those with $\nu=\tau_{orb}=0$. The reader might consult recent works of Viaclovsky and Lock for more examples of self-dual orbifolds.

Of course Theorem \ref{th3} (v) is also a useful tool for computing $\nu_+$. A K\"aher-Einstein 4-orbifold $(M,g)$ of positive scalar curvature achieves $\nu_+(M)$.
For example $M:=(S^2\times S^2)/\Bbb Z_2$ in Example \ref{Pastor-HYJ} has such metric $g$ so that
\begin{eqnarray*}
\nu_+(M)&=&-12\pi^2\tau_{orb}(M)+2\mathcal{W}_+(M,[g])\\ &=& 0+2\frac{4\pi^2}{3}(2\chi_{orb}(M)+3\tau_{orb}(M))\\ &=& \frac{32\pi^2}{3}.
\end{eqnarray*}
But we don't know how to compute $\nu(M)$ for such $M$ due to the paucity of information on $\mathcal{W}_+(M,[g])$ when $Y(M,[g])<0$.

\section{Final remarks}

While it seems impossible to obtain a Gursky-type inequality holding for all conformal classes of negative Yamabe constant on a given 4-manifold,
it's possible to obtain a lower bound of $\mathcal{W}_+$ for conformal classes with negative Yamabe constant, if a lower bound of Yamabe constant is given.
For example H. Seshadri \cite{harish} obtained a sharp inequality
$$\mathcal{W}_+(M,[g])\geq 2\pi^2(\frac{2\chi(M)}{1+c^2/24}+3\tau(M))$$ when the \textit{modified scalar curvature} $s+c|W|$ for a constant $c>0$ is nonnegative, and
similarly M. Itoh \cite{itoh} obtained a sharp inequality
$$\mathcal{W}_+(M,[g])\geq \frac{4\pi^2}{3}(2\chi(M)+3\tau(M))$$
if the modified scalar curvature $s-6w^-$ is nonnegative where $w^-$ is the lowest eigenvalue of $W^+$ pointwisely.

On the other hand, it's noteworthy that LeBrun \cite{LeB2} observed that Gursky's first inequality (i) holds on any conformal class of a del Pezzo surface satisfying that its self-dual harmonic 2-form is nowhere vanishing.

Due to the diverse possibility of orbifold singularities, the topological classification of a canonical geometry on 4-oribifolds is more difficult than that of 4-manifolds, as seen in the case of weighted projective planes. Nevertheless we hope that Theorem \ref{th3} can be further exploited to give some topological restrictions to the existence of canonical metrics of positive scalar curvature on 4-orbifolds.

It would be also interesting to obtain similar estimates of $\mathcal{W}_+$ for an edge-cone metric in view of its great interest in K\"ahler geometry.

\bigskip

\noindent{\bf Acknowledgement.}
We would like to give special thanks to colleagues and students of KNUE where the first idea of this work was conceived and also researchers at KIAS for sharing fellowship through mathematics. It's our pleasure to express sincere thanks to Claude LeBrun for kind suggestions leading to the improvement of the earlier version of this paper, particularly on the residually finite fundamental group and the Lefschetz fixed point theorem. We are also grateful to the anonymous referee for helpful remarks in examples.

\end{document}